\documentclass[11pt,a4paper,reqno]{amsart}
\usepackage[T1]{fontenc}
\usepackage{microtype}
\usepackage[margin=28mm,headheight=14pt]{geometry}
\usepackage{mathtools,enumitem}
\usepackage{xcolor}
\usepackage{tikz}
\usetikzlibrary{arrows.meta}
\usepackage[colorlinks=true,linkcolor=blue!60!black,citecolor=blue!60!black,urlcolor=blue!60!black]{hyperref}
\usepackage[nameinlink,noabbrev,capitalize]{cleveref}
\AtBeginDocument{}
\hypersetup{pdftitle={From boundary geometry to counterexamples in Smale's mean value problem},pdfauthor={Tom Adamczewski and Matthew J. Colbrook},pdfsubject={Geometric mechanisms and polynomial transfer}}
\numberwithin{equation}{section}
\newtheorem{theorem}{Theorem}[section]
\newtheorem{proposition}[theorem]{Proposition}
\newtheorem{lemma}[theorem]{Lemma}
\newtheorem{corollary}[theorem]{Corollary}
\theoremstyle{definition}

\theoremstyle{remark}

\newcommand{\C}{\mathbb C}
\newcommand{\D}{\mathbb D}
\newcommand{\T}{\mathbb T}
\newcommand{\N}{\mathbb N}
\newcommand{\Q}{\mathbb Q}
\newcommand{\ol}[1]{\overline{#1}}

\newcommand{\dist}{\operatorname{dist}}

\newcommand{\wind}{\operatorname{wind}}

\newcommand{\Smal}{\mathcal S}
\newcommand{\Boundary}{\mathcal L}
\newcommand{\dd}{\,\mathrm d}
\DeclareMathOperator{\Ree}{Re}
\DeclareMathOperator{\Imm}{Im}
\setlist[enumerate]{label=\textup{(\roman*)},leftmargin=2.2em,itemsep=3pt,topsep=5pt}
\allowdisplaybreaks[2]
\title[Geometry and counterexamples in Smale's mean value problem]{From boundary geometry to\\counterexamples in Smale's mean value problem}
\author{Tom Adamczewski}
\address{Epoch AI}
\email{tom@epoch.ai}
\author{Matthew J. Colbrook}
\address{Department of Applied Mathematics and Theoretical Physics, University of Cambridge, Cambridge, United Kingdom}
\email{m.colbrook@damtp.cam.ac.uk}
\date{\today}

\begin{document}
\begin{abstract}
This paper provides a human digestion of the recent counterexample to Smale's mean value problem found by GPT-6 Astra. We give a geometric mechanism for constructing counterexamples. An outward normal that re-enters an analytic Jordan domain, subject to a curvature condition, yields a holomorphic model whose normalized divided differences exceed one in modulus on the boundary. Polynomial perturbations then concentrate every critical point near the boundary whilst preserving the strict inequality. This turns boundary geometry into counterexamples in every sufficiently large degree. We also locate these critical points to first order and show that their limiting distribution is the logarithmic equilibrium measure on the boundary. We end with an explicit example of degree $29\,984$, verified by interval arithmetic.\vspace{-5mm}
\end{abstract}
\maketitle

\section{Introduction}

Smale's mean value problem asks for a complex analog of the mean value theorem. For a polynomial $p$, fix $a\in\mathbb{C}$ with $p'(a)\ne 0$. For a critical point $c$, the divided difference
\[
\frac{p(c)-p(a)}{c-a}
\]
is a secant slope. The proposed bound is that at least one critical point $c$ satisfies
\[
\left|\frac{p(c)-p(a)}{c-a}\right|\le |p'(a)|.
\]
To state the problem quantitatively, for $\deg p\ge 2$ define
\begin{equation}\label{eq:smale-functional}
 \Smal(p,a)=\min_{p'(c)=0}
 \left|\frac{p(c)-p(a)}{(c-a)p'(a)}\right|.
\end{equation}
Smale proved that $\Smal(p,a)\le 4$ and conjectured the sharper bound $\Smal(p,a)\le(\deg p-1)/\deg p$ \cite{Smale1981}. We review previous work in \cref{sec:previous}.

This paper provides a human discussion and digestion of a counterexample to the constant-one conjecture ($\Smal(p,a)\le 1$) found by GPT-6 Astra. We will describe the idea behind the counterexample and streamline the process.

\subsection{The main result}

To make the minimum in \labelcref{eq:smale-functional} exceed one, we must beat the proposed bound at every critical point. As an example, consider the normalized polynomials
\[
 p_n(z)=\left(1+\frac{z}{n}\right)^n-1,\qquad  n\ge 2.
\]
These polynomials converge locally uniformly to $e^z-1$, whose derivative never vanishes. Nevertheless, $p_n'$ has a zero at $-n$ of multiplicity $n-1$, and
\[
 \left|\frac{p_n(-n)}{-n}\right|=\frac{1}{n}.
\]
Though the critical point escapes any compact set as $n\rightarrow\infty$, it still determines the minimum in \labelcref{eq:smale-functional}. Hence, it is not enough to approximate a function on a fixed compact. We also need control of the critical points far away.

To obtain this global control, we use the boundary of a domain. We first
construct a holomorphic function whose divided-difference modulus exceeds
$1$ there. Then we approximate the function by polynomials whose critical points
concentrate on that boundary. The following theorem gives a sufficient geometric condition for counterexamples to Smale's mean value conjecture.\footnote{A \emph{Jordan domain} is the interior of a simple closed curve. A boundary is \emph{analytic} if it has a real-analytic parametrization whose derivative does not vanish. We orient $\partial\Omega$ counterclockwise, and write $\nu$ for the outward unit normal and $\kappa$ for the signed curvature.}

\begin{theorem}[Normal-return criterion]\label{thm:normal-main}
Let $\Omega$ be an analytic Jordan domain. Suppose that there exist $\zeta\in\partial\Omega$ and $d>0$ such that
\begin{equation}\label{eq:normal-return}
 a=\zeta+d\cdot\nu(\zeta)\in\Omega,
 \qquad 1+d\cdot\kappa(\zeta)>0.
\end{equation}
Then there exist $\delta>0$ and $n_0\in\N$ such that, for every $n\ge n_0$, there is a degree-$n$ polynomial $p_n$ satisfying
\[
 p_n(0)=0,\qquad p_n'(0)=1,\qquad
 \Smal(p_n,0)\ge 1+\delta.
\]
\end{theorem}

\begin{proof}
{\color{black}\hypersetup{linkcolor=blue!60!black}See \Cref{subsec:normal-proof}.}
\end{proof}

The outward normal ray from $\zeta$ leaves $\Omega$ and then returns to it; see \cref{fig:cubic}. This is impossible for a convex domain. The roles of the curvature at $\zeta$ and the distance $d$ will become clear in \cref{sec:geometry}. The sizes of the degree $n_0$ and margin $\delta$ depend on $\Omega$ and are discussed throughout the proof. \Cref{prop:boundary-universal} shows that $\sup_{n\ge2}\sup_{\deg p=n}\Smal(p,0)$ can be written in terms of a supremum over such domains.

\begin{figure}[htbp]
\centering
\begingroup
\def\CubicBoundaryCoordinates{(1.0000000,0.0000000) (0.8843815,0.0030871) (0.7928512,0.0055352) (0.7234761,0.0075765) (0.6704940,0.0093625) (0.6289684,0.0109787)
(0.5955012,0.0124740) (0.5678510,0.0138780) (0.5445238,0.0152099) (0.5244982,0.0164830) (0.5070557,0.0177068) (0.4916769,0.0188883)
(0.4779768,0.0200332) (0.4656638,0.0211457) (0.4545126,0.0222294) (0.4443463,0.0232872) (0.4350236,0.0243216) (0.4264302,0.0253346)
(0.4184726,0.0263281) (0.4110735,0.0273035) (0.4041680,0.0282622) (0.3977017,0.0292054) (0.3916281,0.0301341) (0.3859075,0.0310494)
(0.3805057,0.0319520) (0.3753931,0.0328426) (0.3705435,0.0337221) (0.3659344,0.0345910) (0.3615457,0.0354499) (0.3573596,0.0362993)
(0.3533604,0.0371397) (0.3495340,0.0379715) (0.3458677,0.0387953) (0.3423503,0.0396113) (0.3389715,0.0404199) (0.3357221,0.0412215)
(0.3325936,0.0420163) (0.3295784,0.0428048) (0.3266695,0.0435870) (0.3238606,0.0443634) (0.3211459,0.0451341) (0.3185199,0.0458994)
(0.3159777,0.0466595) (0.3135147,0.0474146) (0.3111269,0.0481650) (0.3088101,0.0489107) (0.3065609,0.0496520) (0.3043759,0.0503891)
(0.3022520,0.0511221) (0.3001862,0.0518512) (0.2981758,0.0525764) (0.2962183,0.0532981) (0.2943113,0.0540162) (0.2924527,0.0547310)
(0.2906402,0.0554426) (0.2888720,0.0561510) (0.2871462,0.0568565) (0.2854610,0.0575590) (0.2838149,0.0582588) (0.2822063,0.0589560)
(0.2806338,0.0596505) (0.2790958,0.0603427) (0.2775913,0.0610324) (0.2761189,0.0617199) (0.2746774,0.0624052) (0.2732659,0.0630884)
(0.2718831,0.0637696) (0.2705281,0.0644489) (0.2692000,0.0651263) (0.2678978,0.0658020) (0.2666208,0.0664760) (0.2653680,0.0671484)
(0.2641387,0.0678193) (0.2629322,0.0684887) (0.2617477,0.0691567) (0.2605846,0.0698234) (0.2594422,0.0704889) (0.2583200,0.0711532)
(0.2572172,0.0718163) (0.2561333,0.0724784) (0.2550679,0.0731395) (0.2540203,0.0737997) (0.2529901,0.0744590) (0.2519768,0.0751176)
(0.2509799,0.0757753) (0.2499991,0.0764324) (0.2490338,0.0770888) (0.2480836,0.0777447) (0.2471483,0.0784001) (0.2462273,0.0790549)
(0.2453204,0.0797094) (0.2444271,0.0803636) (0.2435472,0.0810174) (0.2426804,0.0816710) (0.2418262,0.0823244) (0.2409845,0.0829776)
(0.2401549,0.0836308) (0.2393371,0.0842839) (0.2385309,0.0849370) (0.2377361,0.0855903) (0.2369523,0.0862436) (0.2361793,0.0868971)
(0.2354169,0.0875508) (0.2346649,0.0882048) (0.2339231,0.0888590) (0.2331912,0.0895137) (0.2324690,0.0901688) (0.2317564,0.0908243)
(0.2310531,0.0914804) (0.2303590,0.0921370) (0.2296738,0.0927943) (0.2289975,0.0934522) (0.2283299,0.0941108) (0.2276707,0.0947702)
(0.2270198,0.0954304) (0.2263772,0.0960914) (0.2257426,0.0967534) (0.2251158,0.0974163) (0.2244968,0.0980803) (0.2238855,0.0987453)
(0.2232816,0.0994114) (0.2226852,0.1000787) (0.2220959,0.1007472) (0.2215138,0.1014170) (0.2209388,0.1020880) (0.2203706,0.1027605)
(0.2198093,0.1034344) (0.2192546,0.1041098) (0.2187066,0.1047866) (0.2181651,0.1054651) (0.2176300,0.1061452) (0.2171012,0.1068270)
(0.2165787,0.1075106) (0.2160623,0.1081959) (0.2155520,0.1088831) (0.2150477,0.1095723) (0.2145493,0.1102634) (0.2140568,0.1109565)
(0.2135700,0.1116517) (0.2130890,0.1123491) (0.2126136,0.1130487) (0.2121438,0.1137505) (0.2116795,0.1144547) (0.2112207,0.1151613)
(0.2107673,0.1158703) (0.2103192,0.1165818) (0.2098765,0.1172960) (0.2094389,0.1180128) (0.2090066,0.1187323) (0.2085795,0.1194546)
(0.2081574,0.1201797) (0.2077404,0.1209078) (0.2073285,0.1216389) (0.2069215,0.1223731) (0.2065195,0.1231103) (0.2061225,0.1238509)
(0.2057303,0.1245947) (0.2053429,0.1253419) (0.2049604,0.1260925) (0.2045827,0.1268467) (0.2042098,0.1276045) (0.2038417,0.1283660)
(0.2034783,0.1291312) (0.2031195,0.1299004) (0.2027655,0.1306735) (0.2024162,0.1314506) (0.2020716,0.1322319) (0.2017316,0.1330175)
(0.2013962,0.1338073) (0.2010655,0.1346016) (0.2007395,0.1354005) (0.2004180,0.1362039) (0.2001012,0.1370121) (0.1997890,0.1378252)
(0.1994815,0.1386432) (0.1991785,0.1394663) (0.1988802,0.1402946) (0.1985866,0.1411281) (0.1982976,0.1419671) (0.1980133,0.1428117)
(0.1977336,0.1436619) (0.1974587,0.1445179) (0.1971884,0.1453798) (0.1969229,0.1462479) (0.1966622,0.1471221) (0.1964062,0.1480027)
(0.1961550,0.1488898) (0.1959087,0.1497835) (0.1956673,0.1506841) (0.1954307,0.1515916) (0.1951991,0.1525063) (0.1949725,0.1534282)
(0.1947509,0.1543576) (0.1945344,0.1552947) (0.1943230,0.1562397) (0.1941167,0.1571926) (0.1939157,0.1581538) (0.1937200,0.1591234)
(0.1935297,0.1601017) (0.1933447,0.1610888) (0.1931653,0.1620849) (0.1929914,0.1630904) (0.1928231,0.1641053) (0.1926606,0.1651301)
(0.1925038,0.1661648) (0.1923530,0.1672099) (0.1922081,0.1682655) (0.1920693,0.1693319) (0.1919367,0.1704094) (0.1918104,0.1714983)
(0.1916906,0.1725990) (0.1915772,0.1737116) (0.1914706,0.1748367) (0.1913707,0.1759744) (0.1912778,0.1771252) (0.1911919,0.1782894)
(0.1911133,0.1794674) (0.1910421,0.1806595) (0.1909785,0.1818663) (0.1909226,0.1830882) (0.1908746,0.1843255) (0.1908348,0.1855788)
(0.1908033,0.1868484) (0.1907804,0.1881351) (0.1907663,0.1894391) (0.1907612,0.1907612) (0.1907654,0.1921018) (0.1907791,0.1934616)
(0.1908026,0.1948412) (0.1908362,0.1962412) (0.1908803,0.1976624) (0.1909351,0.1991053) (0.1910011,0.2005709) (0.1910784,0.2020598)
(0.1911676,0.2035729) (0.1912691,0.2051110) (0.1913832,0.2066751) (0.1915104,0.2082660) (0.1916513,0.2098849) (0.1918062,0.2115328)
(0.1919758,0.2132107) (0.1921605,0.2149199) (0.1923611,0.2166615) (0.1925781,0.2184370) (0.1928122,0.2202477) (0.1930642,0.2220950)
(0.1933348,0.2239805) (0.1936248,0.2259059) (0.1939351,0.2278730) (0.1942667,0.2298835) (0.1946204,0.2319396) (0.1949975,0.2340433)
(0.1953990,0.2361968) (0.1958262,0.2384027) (0.1962804,0.2406635) (0.1967630,0.2429821) (0.1972755,0.2453613) (0.1978197,0.2478044)
(0.1983972,0.2503149) (0.1990100,0.2528966) (0.1996602,0.2555534) (0.2003501,0.2582898) (0.2010822,0.2611106) (0.2018592,0.2640209)
(0.2026840,0.2670263) (0.2035599,0.2701331) (0.2044904,0.2733480) (0.2054795,0.2766784) (0.2065316,0.2801326) (0.2076513,0.2837195)
(0.2088441,0.2874492) (0.2101159,0.2913328) (0.2114733,0.2953828) (0.2129239,0.2996130) (0.2144760,0.3040390) (0.2161391,0.3086786)
(0.2179241,0.3135517) (0.2198433,0.3186810) (0.2219110,0.3240928) (0.2241434,0.3298171) (0.2265598,0.3358887) (0.2291822,0.3423483)
(0.2320371,0.3492439) (0.2351555,0.3566321) (0.2385749,0.3645807) (0.2423405,0.3731716) (0.2465075,0.3825048) (0.2511446,0.3927038)
(0.2563376,0.4039234) (0.2621962,0.4163605) (0.2688624,0.4302697) (0.2765238,0.4459872) (0.2854341,0.4639665) (0.2959442,0.4848344)
(0.3085534,0.5094822) (0.3239920,0.5392132) (0.3433551,0.5759836) (0.3682995,0.6227604) (0.4012193,0.6838617) (0.4448643,0.7643533)
(0.5000000,0.8660254) (0.5619418,0.9812057) (0.6229792,1.0966417) (0.6784977,1.2041394) (0.7275414,1.3017820) (0.7706839,1.3903505)
(0.8088111,1.4712219) (0.8427369,1.5456887) (0.8731277,1.6148158) (0.9005133,1.6794502) (0.9253137,1.7402619) (0.9478641,1.7977843)
(0.9684355,1.8524462) (0.9872489,1.9045962) (1.0044875,1.9545212) (1.0203040,2.0024594) (1.0348279,2.0486110) (1.0481691,2.0931453)
(1.0604225,2.1362073) (1.0716702,2.1779214) (1.0819840,2.2183959) (1.0914269,2.2577249) (1.1000546,2.2959912) (1.1079169,2.3332677)
(1.1150581,2.3696191) (1.1215180,2.4051031) (1.1273326,2.4397713) (1.1325345,2.4736699) (1.1371533,2.5068405) (1.1412161,2.5393211)
(1.1447478,2.5711456) (1.1477710,2.6023453) (1.1503069,2.6329485) (1.1523749,2.6629813) (1.1539930,2.6924673) (1.1551778,2.7214285)
(1.1559452,2.7498849) (1.1563095,2.7778551) (1.1562844,2.8053563) (1.1558828,2.8324044) (1.1551167,2.8590141) (1.1539973,2.8851991)
(1.1525353,2.9109720) (1.1507409,2.9363448) (1.1486236,2.9613285) (1.1461925,2.9859335) (1.1434560,3.0101693) (1.1404225,3.0340450)
(1.1370998,3.0575692) (1.1334952,3.0807496) (1.1296158,3.1035939) (1.1254685,3.1261090) (1.1210597,3.1483014) (1.1163957,3.1701775)
(1.1114825,3.1917430) (1.1063258,3.2130035) (1.1009312,3.2339642) (1.0953040,3.2546299) (1.0894493,3.2750054) (1.0833721,3.2950950)
(1.0770772,3.3149029) (1.0705693,3.3344330) (1.0638528,3.3536890) (1.0569321,3.3726744) (1.0498114,3.3913927) (1.0424948,3.4098469)
(1.0349863,3.4280400) (1.0272898,3.4459750) (1.0194090,3.4636546) (1.0113476,3.4810812) (1.0031092,3.4982574) (0.9946972,3.5151855)
(0.9861152,3.5318677) (0.9773663,3.5483061) (0.9684540,3.5645027) (0.9593812,3.5804595) (0.9501512,3.5961781) (0.9407671,3.6116605)
(0.9312317,3.6269081) (0.9215481,3.6419226) (0.9117191,3.6567054) (0.9017475,3.6712581) (0.8916361,3.6855818) (0.8813876,3.6996781)
(0.8710048,3.7135480) (0.8604902,3.7271927) (0.8498465,3.7406134) (0.8390762,3.7538111) (0.8281818,3.7667869) (0.8171658,3.7795417)
(0.8060307,3.7920765) (0.7947789,3.8043921) (0.7834128,3.8164894) (0.7719346,3.8283692) (0.7603469,3.8400322) (0.7486517,3.8514792)
(0.7368515,3.8627110) (0.7249485,3.8737281) (0.7129449,3.8845311) (0.7008429,3.8951208) (0.6886446,3.9054977) (0.6763523,3.9156623)
(0.6639681,3.9256152) (0.6514941,3.9353569) (0.6389324,3.9448878) (0.6262851,3.9542084) (0.6135542,3.9633191) (0.6007419,3.9722204)
(0.5878500,3.9809127) (0.5748808,3.9893963) (0.5618361,3.9976716) (0.5487180,4.0057389) (0.5355284,4.0135986) (0.5222693,4.0212510)
(0.5089427,4.0286963) (0.4955505,4.0359349) (0.4820946,4.0429670) (0.4685770,4.0497929) (0.4549996,4.0564128) (0.4413642,4.0628269)
(0.4276729,4.0690356) (0.4139274,4.0750389) (0.4001296,4.0808370) (0.3862814,4.0864303) (0.3723847,4.0918187) (0.3584413,4.0970026)
(0.3444531,4.1019820) (0.3304218,4.1067571) (0.3163495,4.1113281) (0.3022378,4.1156951) (0.2880885,4.1198582) (0.2739037,4.1238175)
(0.2596849,4.1275731) (0.2454341,4.1311252) (0.2311530,4.1344738) (0.2168434,4.1376190) (0.2025072,4.1405609) (0.1881461,4.1432997)
(0.1737620,4.1458352) (0.1593565,4.1481677) (0.1449316,4.1502972) (0.1304889,4.1522237) (0.1160303,4.1539473) (0.1015575,4.1554680)
(0.0870723,4.1567859) (0.0725764,4.1579010) (0.0580718,4.1588133) (0.0435600,4.1595228) (0.0290430,4.1600296) (0.0145224,4.1603337)
(0.0000000,4.1604351) (-0.0145224,4.1603337) (-0.0290430,4.1600296) (-0.0435600,4.1595228) (-0.0580718,4.1588133) (-0.0725764,4.1579010)
(-0.0870723,4.1567859) (-0.1015575,4.1554680) (-0.1160303,4.1539473) (-0.1304889,4.1522237) (-0.1449316,4.1502972) (-0.1593565,4.1481677)
(-0.1737620,4.1458352) (-0.1881461,4.1432997) (-0.2025072,4.1405609) (-0.2168434,4.1376190) (-0.2311530,4.1344738) (-0.2454341,4.1311252)
(-0.2596849,4.1275731) (-0.2739037,4.1238175) (-0.2880885,4.1198582) (-0.3022378,4.1156951) (-0.3163495,4.1113281) (-0.3304218,4.1067571)
(-0.3444531,4.1019820) (-0.3584413,4.0970026) (-0.3723847,4.0918187) (-0.3862814,4.0864303) (-0.4001296,4.0808370) (-0.4139274,4.0750389)
(-0.4276729,4.0690356) (-0.4413642,4.0628269) (-0.4549996,4.0564128) (-0.4685770,4.0497929) (-0.4820946,4.0429670) (-0.4955505,4.0359349)
(-0.5089427,4.0286963) (-0.5222693,4.0212510) (-0.5355284,4.0135986) (-0.5487180,4.0057389) (-0.5618361,3.9976716) (-0.5748808,3.9893963)
(-0.5878500,3.9809127) (-0.6007419,3.9722204) (-0.6135542,3.9633191) (-0.6262851,3.9542084) (-0.6389324,3.9448878) (-0.6514941,3.9353569)
(-0.6639681,3.9256152) (-0.6763523,3.9156623) (-0.6886446,3.9054977) (-0.7008429,3.8951208) (-0.7129449,3.8845311) (-0.7249485,3.8737281)
(-0.7368515,3.8627110) (-0.7486517,3.8514792) (-0.7603469,3.8400322) (-0.7719346,3.8283692) (-0.7834128,3.8164894) (-0.7947789,3.8043921)
(-0.8060307,3.7920765) (-0.8171658,3.7795417) (-0.8281818,3.7667869) (-0.8390762,3.7538111) (-0.8498465,3.7406134) (-0.8604902,3.7271927)
(-0.8710048,3.7135480) (-0.8813876,3.6996781) (-0.8916361,3.6855818) (-0.9017475,3.6712581) (-0.9117191,3.6567054) (-0.9215481,3.6419226)
(-0.9312317,3.6269081) (-0.9407671,3.6116605) (-0.9501512,3.5961781) (-0.9593812,3.5804595) (-0.9684540,3.5645027) (-0.9773663,3.5483061)
(-0.9861152,3.5318677) (-0.9946972,3.5151855) (-1.0031092,3.4982574) (-1.0113476,3.4810812) (-1.0194090,3.4636546) (-1.0272898,3.4459750)
(-1.0349863,3.4280400) (-1.0424948,3.4098469) (-1.0498114,3.3913927) (-1.0569321,3.3726744) (-1.0638528,3.3536890) (-1.0705693,3.3344330)
(-1.0770772,3.3149029) (-1.0833721,3.2950950) (-1.0894493,3.2750054) (-1.0953040,3.2546299) (-1.1009312,3.2339642) (-1.1063258,3.2130035)
(-1.1114825,3.1917430) (-1.1163957,3.1701775) (-1.1210597,3.1483014) (-1.1254685,3.1261090) (-1.1296158,3.1035939) (-1.1334952,3.0807496)
(-1.1370998,3.0575692) (-1.1404225,3.0340450) (-1.1434560,3.0101693) (-1.1461925,2.9859335) (-1.1486236,2.9613285) (-1.1507409,2.9363448)
(-1.1525353,2.9109720) (-1.1539973,2.8851991) (-1.1551167,2.8590141) (-1.1558828,2.8324044) (-1.1562844,2.8053563) (-1.1563095,2.7778551)
(-1.1559452,2.7498849) (-1.1551778,2.7214285) (-1.1539930,2.6924673) (-1.1523749,2.6629813) (-1.1503069,2.6329485) (-1.1477710,2.6023453)
(-1.1447478,2.5711456) (-1.1412161,2.5393211) (-1.1371533,2.5068405) (-1.1325345,2.4736699) (-1.1273326,2.4397713) (-1.1215180,2.4051031)
(-1.1150581,2.3696191) (-1.1079169,2.3332677) (-1.1000546,2.2959912) (-1.0914269,2.2577249) (-1.0819840,2.2183959) (-1.0716702,2.1779214)
(-1.0604225,2.1362073) (-1.0481691,2.0931453) (-1.0348279,2.0486110) (-1.0203040,2.0024594) (-1.0044875,1.9545212) (-0.9872489,1.9045962)
(-0.9684355,1.8524462) (-0.9478641,1.7977843) (-0.9253137,1.7402619) (-0.9005133,1.6794502) (-0.8731277,1.6148158) (-0.8427369,1.5456887)
(-0.8088111,1.4712219) (-0.7706839,1.3903505) (-0.7275414,1.3017820) (-0.6784977,1.2041394) (-0.6229792,1.0966417) (-0.5619418,0.9812057)
(-0.5000000,0.8660254) (-0.4448643,0.7643533) (-0.4012193,0.6838617) (-0.3682995,0.6227604) (-0.3433551,0.5759836) (-0.3239920,0.5392132)
(-0.3085534,0.5094822) (-0.2959442,0.4848344) (-0.2854341,0.4639665) (-0.2765238,0.4459872) (-0.2688624,0.4302697) (-0.2621962,0.4163605)
(-0.2563376,0.4039234) (-0.2511446,0.3927038) (-0.2465075,0.3825048) (-0.2423405,0.3731716) (-0.2385749,0.3645807) (-0.2351555,0.3566321)
(-0.2320371,0.3492439) (-0.2291822,0.3423483) (-0.2265598,0.3358887) (-0.2241434,0.3298171) (-0.2219110,0.3240928) (-0.2198433,0.3186810)
(-0.2179241,0.3135517) (-0.2161391,0.3086786) (-0.2144760,0.3040390) (-0.2129239,0.2996130) (-0.2114733,0.2953828) (-0.2101159,0.2913328)
(-0.2088441,0.2874492) (-0.2076513,0.2837195) (-0.2065316,0.2801326) (-0.2054795,0.2766784) (-0.2044904,0.2733480) (-0.2035599,0.2701331)
(-0.2026840,0.2670263) (-0.2018592,0.2640209) (-0.2010822,0.2611106) (-0.2003501,0.2582898) (-0.1996602,0.2555534) (-0.1990100,0.2528966)
(-0.1983972,0.2503149) (-0.1978197,0.2478044) (-0.1972755,0.2453613) (-0.1967630,0.2429821) (-0.1962804,0.2406635) (-0.1958262,0.2384027)
(-0.1953990,0.2361968) (-0.1949975,0.2340433) (-0.1946204,0.2319396) (-0.1942667,0.2298835) (-0.1939351,0.2278730) (-0.1936248,0.2259059)
(-0.1933348,0.2239805) (-0.1930642,0.2220950) (-0.1928122,0.2202477) (-0.1925781,0.2184370) (-0.1923611,0.2166615) (-0.1921605,0.2149199)
(-0.1919758,0.2132107) (-0.1918062,0.2115328) (-0.1916513,0.2098849) (-0.1915104,0.2082660) (-0.1913832,0.2066751) (-0.1912691,0.2051110)
(-0.1911676,0.2035729) (-0.1910784,0.2020598) (-0.1910011,0.2005709) (-0.1909351,0.1991053) (-0.1908803,0.1976624) (-0.1908362,0.1962412)
(-0.1908026,0.1948412) (-0.1907791,0.1934616) (-0.1907654,0.1921018) (-0.1907612,0.1907612) (-0.1907663,0.1894391) (-0.1907804,0.1881351)
(-0.1908033,0.1868484) (-0.1908348,0.1855788) (-0.1908746,0.1843255) (-0.1909226,0.1830882) (-0.1909785,0.1818663) (-0.1910421,0.1806595)
(-0.1911133,0.1794674) (-0.1911919,0.1782894) (-0.1912778,0.1771252) (-0.1913707,0.1759744) (-0.1914706,0.1748367) (-0.1915772,0.1737116)
(-0.1916906,0.1725990) (-0.1918104,0.1714983) (-0.1919367,0.1704094) (-0.1920693,0.1693319) (-0.1922081,0.1682655) (-0.1923530,0.1672099)
(-0.1925038,0.1661648) (-0.1926606,0.1651301) (-0.1928231,0.1641053) (-0.1929914,0.1630904) (-0.1931653,0.1620849) (-0.1933447,0.1610888)
(-0.1935297,0.1601017) (-0.1937200,0.1591234) (-0.1939157,0.1581538) (-0.1941167,0.1571926) (-0.1943230,0.1562397) (-0.1945344,0.1552947)
(-0.1947509,0.1543576) (-0.1949725,0.1534282) (-0.1951991,0.1525063) (-0.1954307,0.1515916) (-0.1956673,0.1506841) (-0.1959087,0.1497835)
(-0.1961550,0.1488898) (-0.1964062,0.1480027) (-0.1966622,0.1471221) (-0.1969229,0.1462479) (-0.1971884,0.1453798) (-0.1974587,0.1445179)
(-0.1977336,0.1436619) (-0.1980133,0.1428117) (-0.1982976,0.1419671) (-0.1985866,0.1411281) (-0.1988802,0.1402946) (-0.1991785,0.1394663)
(-0.1994815,0.1386432) (-0.1997890,0.1378252) (-0.2001012,0.1370121) (-0.2004180,0.1362039) (-0.2007395,0.1354005) (-0.2010655,0.1346016)
(-0.2013962,0.1338073) (-0.2017316,0.1330175) (-0.2020716,0.1322319) (-0.2024162,0.1314506) (-0.2027655,0.1306735) (-0.2031195,0.1299004)
(-0.2034783,0.1291312) (-0.2038417,0.1283660) (-0.2042098,0.1276045) (-0.2045827,0.1268467) (-0.2049604,0.1260925) (-0.2053429,0.1253419)
(-0.2057303,0.1245947) (-0.2061225,0.1238509) (-0.2065195,0.1231103) (-0.2069215,0.1223731) (-0.2073285,0.1216389) (-0.2077404,0.1209078)
(-0.2081574,0.1201797) (-0.2085795,0.1194546) (-0.2090066,0.1187323) (-0.2094389,0.1180128) (-0.2098765,0.1172960) (-0.2103192,0.1165818)
(-0.2107673,0.1158703) (-0.2112207,0.1151613) (-0.2116795,0.1144547) (-0.2121438,0.1137505) (-0.2126136,0.1130487) (-0.2130890,0.1123491)
(-0.2135700,0.1116517) (-0.2140568,0.1109565) (-0.2145493,0.1102634) (-0.2150477,0.1095723) (-0.2155520,0.1088831) (-0.2160623,0.1081959)
(-0.2165787,0.1075106) (-0.2171012,0.1068270) (-0.2176300,0.1061452) (-0.2181651,0.1054651) (-0.2187066,0.1047866) (-0.2192546,0.1041098)
(-0.2198093,0.1034344) (-0.2203706,0.1027605) (-0.2209388,0.1020880) (-0.2215138,0.1014170) (-0.2220959,0.1007472) (-0.2226852,0.1000787)
(-0.2232816,0.0994114) (-0.2238855,0.0987453) (-0.2244968,0.0980803) (-0.2251158,0.0974163) (-0.2257426,0.0967534) (-0.2263772,0.0960914)
(-0.2270198,0.0954304) (-0.2276707,0.0947702) (-0.2283299,0.0941108) (-0.2289975,0.0934522) (-0.2296738,0.0927943) (-0.2303590,0.0921370)
(-0.2310531,0.0914804) (-0.2317564,0.0908243) (-0.2324690,0.0901688) (-0.2331912,0.0895137) (-0.2339231,0.0888590) (-0.2346649,0.0882048)
(-0.2354169,0.0875508) (-0.2361793,0.0868971) (-0.2369523,0.0862436) (-0.2377361,0.0855903) (-0.2385309,0.0849370) (-0.2393371,0.0842839)
(-0.2401549,0.0836308) (-0.2409845,0.0829776) (-0.2418262,0.0823244) (-0.2426804,0.0816710) (-0.2435472,0.0810174) (-0.2444271,0.0803636)
(-0.2453204,0.0797094) (-0.2462273,0.0790549) (-0.2471483,0.0784001) (-0.2480836,0.0777447) (-0.2490338,0.0770888) (-0.2499991,0.0764324)
(-0.2509799,0.0757753) (-0.2519768,0.0751176) (-0.2529901,0.0744590) (-0.2540203,0.0737997) (-0.2550679,0.0731395) (-0.2561333,0.0724784)
(-0.2572172,0.0718163) (-0.2583200,0.0711532) (-0.2594422,0.0704889) (-0.2605846,0.0698234) (-0.2617477,0.0691567) (-0.2629322,0.0684887)
(-0.2641387,0.0678193) (-0.2653680,0.0671484) (-0.2666208,0.0664760) (-0.2678978,0.0658020) (-0.2692000,0.0651263) (-0.2705281,0.0644489)
(-0.2718831,0.0637696) (-0.2732659,0.0630884) (-0.2746774,0.0624052) (-0.2761189,0.0617199) (-0.2775913,0.0610324) (-0.2790958,0.0603427)
(-0.2806338,0.0596505) (-0.2822063,0.0589560) (-0.2838149,0.0582588) (-0.2854610,0.0575590) (-0.2871462,0.0568565) (-0.2888720,0.0561510)
(-0.2906402,0.0554426) (-0.2924527,0.0547310) (-0.2943113,0.0540162) (-0.2962183,0.0532981) (-0.2981758,0.0525764) (-0.3001862,0.0518512)
(-0.3022520,0.0511221) (-0.3043759,0.0503891) (-0.3065609,0.0496520) (-0.3088101,0.0489107) (-0.3111269,0.0481650) (-0.3135147,0.0474146)
(-0.3159777,0.0466595) (-0.3185199,0.0458994) (-0.3211459,0.0451341) (-0.3238606,0.0443634) (-0.3266695,0.0435870) (-0.3295784,0.0428048)
(-0.3325936,0.0420163) (-0.3357221,0.0412215) (-0.3389715,0.0404199) (-0.3423503,0.0396113) (-0.3458677,0.0387953) (-0.3495340,0.0379715)
(-0.3533604,0.0371397) (-0.3573596,0.0362993) (-0.3615457,0.0354499) (-0.3659344,0.0345910) (-0.3705435,0.0337221) (-0.3753931,0.0328426)
(-0.3805057,0.0319520) (-0.3859075,0.0310494) (-0.3916281,0.0301341) (-0.3977017,0.0292054) (-0.4041680,0.0282622) (-0.4110735,0.0273035)
(-0.4184726,0.0263281) (-0.4264302,0.0253346) (-0.4350236,0.0243216) (-0.4443463,0.0232872) (-0.4545126,0.0222294) (-0.4656638,0.0211457)
(-0.4779768,0.0200332) (-0.4916769,0.0188883) (-0.5070557,0.0177068) (-0.5244982,0.0164830) (-0.5445238,0.0152099) (-0.5678510,0.0138780)
(-0.5955012,0.0124740) (-0.6289684,0.0109787) (-0.6704940,0.0093625) (-0.7234761,0.0075765) (-0.7928512,0.0055352) (-0.8843815,0.0030871)
(-1.0000000,0.0000000) (-1.1307199,-0.0039470) (-1.2612092,-0.0088050) (-1.3820642,-0.0144735) (-1.4911470,-0.0208217) (-1.5894208,-0.0277434)
(-1.6785211,-0.0351600) (-1.7599741,-0.0430128) (-1.8350354,-0.0512572) (-1.9047032,-0.0598577) (-1.9697679,-0.0687858) (-2.0308589,-0.0780177)
(-2.0884832,-0.0875334) (-2.1430532,-0.0973155) (-2.1949087,-0.1073489) (-2.2443328,-0.1176205) (-2.2915631,-0.1281183) (-2.3368016,-0.1388316)
(-2.3802210,-0.1497508) (-2.4219704,-0.1608671) (-2.4621792,-0.1721723) (-2.5009605,-0.1836591) (-2.5384140,-0.1953203) (-2.5746275,-0.2071496)
(-2.6096794,-0.2191409) (-2.6436394,-0.2312885) (-2.6765702,-0.2435870) (-2.7085282,-0.2560313) (-2.7395642,-0.2686166) (-2.7697246,-0.2813384)
(-2.7990513,-0.2941921) (-2.8275827,-0.3071738) (-2.8553538,-0.3202792) (-2.8823969,-0.3335047) (-2.9087416,-0.3468464) (-2.9344151,-0.3603009)
(-2.9594427,-0.3738646) (-2.9838478,-0.3875342) (-3.0076521,-0.4013065) (-3.0308756,-0.4151783) (-3.0535372,-0.4291467) (-3.0756543,-0.4432086)
(-3.0972434,-0.4573611) (-3.1183197,-0.4716015) (-3.1388975,-0.4859270) (-3.1589905,-0.5003349) (-3.1786111,-0.5148227) (-3.1977713,-0.5293876)
(-3.2164825,-0.5440273) (-3.2347550,-0.5587392) (-3.2525991,-0.5735210) (-3.2700240,-0.5883702) (-3.2870389,-0.6032845) (-3.3036521,-0.6182617)
(-3.3198718,-0.6332994) (-3.3357056,-0.6483955) (-3.3511608,-0.6635477) (-3.3662442,-0.6787539) (-3.3809626,-0.6940119) (-3.3953221,-0.7093197)
(-3.4093288,-0.7246752) (-3.4229883,-0.7400763) (-3.4363063,-0.7555209) (-3.4492878,-0.7710072) (-3.4619379,-0.7865330) (-3.4742614,-0.8020964)
(-3.4862629,-0.8176956) (-3.4979468,-0.8333285) (-3.5093173,-0.8489932) (-3.5203786,-0.8646879) (-3.5311344,-0.8804107) (-3.5415886,-0.8961597)
(-3.5517448,-0.9119331) (-3.5616064,-0.9277290) (-3.5711769,-0.9435456) (-3.5804595,-0.9593812) (-3.5894572,-0.9752339) (-3.5981732,-0.9911020)
(-3.6066104,-1.0069837) (-3.6147715,-1.0228772) (-3.6226593,-1.0387808) (-3.6302765,-1.0546928) (-3.6376255,-1.0706114) (-3.6447090,-1.0865349)
(-3.6515293,-1.1024617) (-3.6580887,-1.1183899) (-3.6643895,-1.1343180) (-3.6704339,-1.1502443) (-3.6762241,-1.1661670) (-3.6817621,-1.1820845)
(-3.6870499,-1.1979951) (-3.6920897,-1.2138973) (-3.6968831,-1.2297893) (-3.7014323,-1.2456696) (-3.7057389,-1.2615364) (-3.7098047,-1.2773882)
(-3.7136316,-1.2932233) (-3.7172212,-1.3090402) (-3.7205751,-1.3248372) (-3.7236950,-1.3406127) (-3.7265825,-1.3563651) (-3.7292392,-1.3720929)
(-3.7316665,-1.3877943) (-3.7338661,-1.4034680) (-3.7358392,-1.4191122) (-3.7375874,-1.4347254) (-3.7391121,-1.4503060) (-3.7404147,-1.4658525)
(-3.7414965,-1.4813633) (-3.7423589,-1.4968368) (-3.7430032,-1.5122714) (-3.7434306,-1.5276657) (-3.7436425,-1.5430181) (-3.7436401,-1.5583270)
(-3.7434247,-1.5735908) (-3.7429974,-1.5888081) (-3.7423594,-1.6039773) (-3.7415120,-1.6190968) (-3.7404563,-1.6341652) (-3.7391935,-1.6491808)
(-3.7377246,-1.6641422) (-3.7360509,-1.6790478) (-3.7341733,-1.6938961) (-3.7320931,-1.7086856) (-3.7298113,-1.7234148) (-3.7273290,-1.7380820)
(-3.7246472,-1.7526859) (-3.7217669,-1.7672249) (-3.7186893,-1.7816974) (-3.7154154,-1.7961020) (-3.7119461,-1.8104371) (-3.7082825,-1.8247012)
(-3.7044256,-1.8388928) (-3.7003764,-1.8530105) (-3.6961358,-1.8670525) (-3.6917049,-1.8810176) (-3.6870846,-1.8949041) (-3.6822758,-1.9087105)
(-3.6772796,-1.9224353) (-3.6720969,-1.9360771) (-3.6667286,-1.9496342) (-3.6611757,-1.9631052) (-3.6554390,-1.9764885) (-3.6495196,-1.9897827)
(-3.6434183,-2.0029862) (-3.6371361,-2.0160975) (-3.6306738,-2.0291150) (-3.6240325,-2.0420373) (-3.6172128,-2.0548629) (-3.6102159,-2.0675901)
(-3.6030425,-2.0802175) (-3.5956935,-2.0927436) (-3.5881699,-2.1051668) (-3.5804724,-2.1174855) (-3.5726021,-2.1296983) (-3.5645597,-2.1418035)
(-3.5563461,-2.1537997) (-3.5479622,-2.1656854) (-3.5394088,-2.1774588) (-3.5306868,-2.1891186) (-3.5217970,-2.2006630) (-3.5127404,-2.2120907)
(-3.5035177,-2.2233999) (-3.4941297,-2.2345892) (-3.4845774,-2.2456569) (-3.4748615,-2.2566014) (-3.4649828,-2.2674212) (-3.4549423,-2.2781147)
(-3.4447407,-2.2886803) (-3.4343789,-2.2991163) (-3.4238576,-2.3094211) (-3.4131776,-2.3195931) (-3.4023399,-2.3296307) (-3.3913451,-2.3395323)
(-3.3801941,-2.3492961) (-3.3688877,-2.3589206) (-3.3574266,-2.3684040) (-3.3458117,-2.3777446) (-3.3340438,-2.3869409) (-3.3221235,-2.3959910)
(-3.3100517,-2.4048934) (-3.2978292,-2.4136461) (-3.2854567,-2.4222476) (-3.2729350,-2.4306960) (-3.2602648,-2.4389897) (-3.2474469,-2.4471268)
(-3.2344820,-2.4551055) (-3.2213708,-2.4629240) (-3.2081141,-2.4705805) (-3.1947127,-2.4780732) (-3.1811671,-2.4854001) (-3.1674781,-2.4925595)
(-3.1536465,-2.4995494) (-3.1396729,-2.5063679) (-3.1255579,-2.5130131) (-3.1113024,-2.5194830) (-3.0969068,-2.5257756) (-3.0823720,-2.5318889)
(-3.0676984,-2.5378209) (-3.0528869,-2.5435695) (-3.0379379,-2.5491326) (-3.0228522,-2.5545081) (-3.0076302,-2.5596939) (-2.9922727,-2.5646878)
(-2.9767801,-2.5694876) (-2.9611530,-2.5740910) (-2.9453920,-2.5784958) (-2.9294976,-2.5826996) (-2.9134704,-2.5867000) (-2.8973107,-2.5904948)
(-2.8810192,-2.5940813) (-2.8645962,-2.5974572) (-2.8480422,-2.6006199) (-2.8313577,-2.6035669) (-2.8145430,-2.6062954) (-2.7975984,-2.6088027)
(-2.7805245,-2.6110863) (-2.7633214,-2.6131431) (-2.7459895,-2.6149704) (-2.7285290,-2.6165652) (-2.7109403,-2.6179246) (-2.6932234,-2.6190453)
(-2.6753787,-2.6199244) (-2.6574062,-2.6205584) (-2.6393060,-2.6209442) (-2.6210783,-2.6210783) (-2.6027229,-2.6209571) (-2.5842401,-2.6205771)
(-2.5656296,-2.6199347) (-2.5468913,-2.6190258) (-2.5280252,-2.6178467) (-2.5090310,-2.6163933) (-2.4899084,-2.6146613) (-2.4706570,-2.6126466)
(-2.4512766,-2.6103444) (-2.4317666,-2.6077504) (-2.4121265,-2.6048597) (-2.3923557,-2.6016673) (-2.3724534,-2.5981680) (-2.3524190,-2.5943567)
(-2.3322515,-2.5902277) (-2.3119499,-2.5857753) (-2.2915133,-2.5809935) (-2.2709402,-2.5758761) (-2.2502295,-2.5704165) (-2.2293798,-2.5646080)
(-2.2083893,-2.5584436) (-2.1872564,-2.5519158) (-2.1659792,-2.5450169) (-2.1445555,-2.5377388) (-2.1229833,-2.5300729) (-2.1012599,-2.5220104)
(-2.0793827,-2.5135419) (-2.0573488,-2.5046574) (-2.0351550,-2.4953466) (-2.0127980,-2.4855985) (-1.9902739,-2.4754015) (-1.9675787,-2.4647433)
(-1.9447080,-2.4536109) (-1.9216570,-2.4419905) (-1.8984205,-2.4298674) (-1.8749928,-2.4172261) (-1.8513676,-2.4040499) (-1.8275384,-2.3903209)
(-1.8034976,-2.3760203) (-1.7792373,-2.3611276) (-1.7547486,-2.3456209) (-1.7300220,-2.3294766) (-1.7050469,-2.3126693) (-1.6798116,-2.2951715)
(-1.6543035,-2.2769535) (-1.6285085,-2.2579827) (-1.6024109,-2.2382239) (-1.5759937,-2.2176386) (-1.5492376,-2.1961843) (-1.5221214,-2.1738147)
(-1.4946213,-2.1504782) (-1.4667106,-2.1261180) (-1.4383593,-2.1006708) (-1.4095337,-2.0740658) (-1.3801952,-2.0462236) (-1.3503002,-2.0170543)
(-1.3197985,-1.9864564) (-1.2886325,-1.9543137) (-1.2567352,-1.9204927) (-1.2240287,-1.8848389) (-1.1904213,-1.8471723) (-1.1558043,-1.8072808)
(-1.1200477,-1.7649128) (-1.0829948,-1.7197666) (-1.0444542,-1.6714761) (-1.0041899,-1.6195925) (-0.9619077,-1.5635587) (-0.9172373,-1.5026759)
(-0.8697100,-1.4360619) (-0.8187369,-1.3626071) (-0.7636056,-1.2809603) (-0.7035665,-1.1896660) (-0.6382300,-1.0878367) (-0.5687781,-0.9772587)
(-0.5000000,-0.8660254) (-0.4395173,-0.7674404) (-0.3916320,-0.6893969) (-0.3551766,-0.6303369) (-0.3271388,-0.5853460) (-0.3049764,-0.5501919)
(-0.2869478,-0.5219562) (-0.2719068,-0.4987124) (-0.2590897,-0.4791764) (-0.2479744,-0.4624703) (-0.2381933,-0.4479765) (-0.2294807,-0.4352488)
(-0.2216392,-0.4239566) (-0.2145192,-0.4138495) (-0.2080051,-0.4047342) (-0.2020058,-0.3964588) (-0.1964487,-0.3889023) (-0.1912747,-0.3819667)
(-0.1864355,-0.3755720) (-0.1818912,-0.3696518) (-0.1776082,-0.3641508) (-0.1735582,-0.3590224) (-0.1697171,-0.3542269) (-0.1660642,-0.3497304)
(-0.1625817,-0.3455036) (-0.1592540,-0.3415212) (-0.1560676,-0.3377611) (-0.1530105,-0.3342040) (-0.1500724,-0.3308327) (-0.1472437,-0.3276321)
(-0.1445163,-0.3245889) (-0.1418827,-0.3216910) (-0.1393362,-0.3189279) (-0.1368708,-0.3162897) (-0.1344811,-0.3137679) (-0.1321622,-0.3113546)
(-0.1299096,-0.3090426) (-0.1277192,-0.3068256) (-0.1255873,-0.3046976) (-0.1235105,-0.3026532) (-0.1214856,-0.3006875) (-0.1195099,-0.2987960)
(-0.1175805,-0.2969745) (-0.1156951,-0.2952191) (-0.1138513,-0.2935263) (-0.1120471,-0.2918928) (-0.1102805,-0.2903156) (-0.1085497,-0.2887918)
(-0.1068530,-0.2873189) (-0.1051887,-0.2858944) (-0.1035554,-0.2845160) (-0.1019517,-0.2831816) (-0.1003762,-0.2818892) (-0.0988279,-0.2806370)
(-0.0973054,-0.2794231) (-0.0958078,-0.2782460) (-0.0943339,-0.2771041) (-0.0928829,-0.2759960) (-0.0914538,-0.2749204) (-0.0900458,-0.2738758)
(-0.0886580,-0.2728612) (-0.0872896,-0.2718754) (-0.0859400,-0.2709173) (-0.0846085,-0.2699859) (-0.0832942,-0.2690802) (-0.0819968,-0.2681994)
(-0.0807154,-0.2673424) (-0.0794497,-0.2665087) (-0.0781989,-0.2656972) (-0.0769627,-0.2649073) (-0.0757405,-0.2641384) (-0.0745318,-0.2633896)
(-0.0733361,-0.2626605) (-0.0721531,-0.2619503) (-0.0709824,-0.2612586) (-0.0698234,-0.2605846) (-0.0686759,-0.2599280) (-0.0675395,-0.2592882)
(-0.0664138,-0.2586648) (-0.0652985,-0.2580572) (-0.0641932,-0.2574650) (-0.0630977,-0.2568879) (-0.0620116,-0.2563254) (-0.0609347,-0.2557771)
(-0.0598666,-0.2552427) (-0.0588071,-0.2547217) (-0.0577560,-0.2542140) (-0.0567129,-0.2537191) (-0.0556777,-0.2532367) (-0.0546501,-0.2527666)
(-0.0536298,-0.2523084) (-0.0526167,-0.2518619) (-0.0516105,-0.2514268) (-0.0506111,-0.2510029) (-0.0496181,-0.2505898) (-0.0486315,-0.2501875)
(-0.0476511,-0.2497956) (-0.0466766,-0.2494140) (-0.0457078,-0.2490424) (-0.0447447,-0.2486806) (-0.0437870,-0.2483285) (-0.0428346,-0.2479858)
(-0.0418873,-0.2476524) (-0.0409449,-0.2473282) (-0.0400074,-0.2470129) (-0.0390744,-0.2467063) (-0.0381460,-0.2464085) (-0.0372220,-0.2461191)
(-0.0363022,-0.2458380) (-0.0353865,-0.2455652) (-0.0344747,-0.2453005) (-0.0335668,-0.2450438) (-0.0326626,-0.2447948) (-0.0317620,-0.2445537)
(-0.0308648,-0.2443201) (-0.0299710,-0.2440941) (-0.0290804,-0.2438755) (-0.0281929,-0.2436642) (-0.0273084,-0.2434601) (-0.0264268,-0.2432632)
(-0.0255480,-0.2430733) (-0.0246719,-0.2428904) (-0.0237983,-0.2427143) (-0.0229273,-0.2425451) (-0.0220585,-0.2423826) (-0.0211921,-0.2422268)
(-0.0203278,-0.2420776) (-0.0194656,-0.2419350) (-0.0186054,-0.2417988) (-0.0177471,-0.2416691) (-0.0168905,-0.2415457) (-0.0160357,-0.2414287)
(-0.0151824,-0.2413179) (-0.0143307,-0.2412134) (-0.0134804,-0.2411151) (-0.0126315,-0.2410229) (-0.0117838,-0.2409368) (-0.0109372,-0.2408569)
(-0.0100918,-0.2407829) (-0.0092473,-0.2407150) (-0.0084038,-0.2406531) (-0.0075611,-0.2405972) (-0.0067191,-0.2405472) (-0.0058778,-0.2405031)
(-0.0050370,-0.2404650) (-0.0041968,-0.2404327) (-0.0033569,-0.2404063) (-0.0025174,-0.2403858) (-0.0016781,-0.2403712) (-0.0008390,-0.2403624)
(-0.0000000,-0.2403595) (0.0008390,-0.2403624) (0.0016781,-0.2403712) (0.0025174,-0.2403858) (0.0033569,-0.2404063) (0.0041968,-0.2404327)
(0.0050370,-0.2404650) (0.0058778,-0.2405031) (0.0067191,-0.2405472) (0.0075611,-0.2405972) (0.0084038,-0.2406531) (0.0092473,-0.2407150)
(0.0100918,-0.2407829) (0.0109372,-0.2408569) (0.0117838,-0.2409368) (0.0126315,-0.2410229) (0.0134804,-0.2411151) (0.0143307,-0.2412134)
(0.0151824,-0.2413179) (0.0160357,-0.2414287) (0.0168905,-0.2415457) (0.0177471,-0.2416691) (0.0186054,-0.2417988) (0.0194656,-0.2419350)
(0.0203278,-0.2420776) (0.0211921,-0.2422268) (0.0220585,-0.2423826) (0.0229273,-0.2425451) (0.0237983,-0.2427143) (0.0246719,-0.2428904)
(0.0255480,-0.2430733) (0.0264268,-0.2432632) (0.0273084,-0.2434601) (0.0281929,-0.2436642) (0.0290804,-0.2438755) (0.0299710,-0.2440941)
(0.0308648,-0.2443201) (0.0317620,-0.2445537) (0.0326626,-0.2447948) (0.0335668,-0.2450438) (0.0344747,-0.2453005) (0.0353865,-0.2455652)
(0.0363022,-0.2458380) (0.0372220,-0.2461191) (0.0381460,-0.2464085) (0.0390744,-0.2467063) (0.0400074,-0.2470129) (0.0409449,-0.2473282)
(0.0418873,-0.2476524) (0.0428346,-0.2479858) (0.0437870,-0.2483285) (0.0447447,-0.2486806) (0.0457078,-0.2490424) (0.0466766,-0.2494140)
(0.0476511,-0.2497956) (0.0486315,-0.2501875) (0.0496181,-0.2505898) (0.0506111,-0.2510029) (0.0516105,-0.2514268) (0.0526167,-0.2518619)
(0.0536298,-0.2523084) (0.0546501,-0.2527666) (0.0556777,-0.2532367) (0.0567129,-0.2537191) (0.0577560,-0.2542140) (0.0588071,-0.2547217)
(0.0598666,-0.2552427) (0.0609347,-0.2557771) (0.0620116,-0.2563254) (0.0630977,-0.2568879) (0.0641932,-0.2574650) (0.0652985,-0.2580572)
(0.0664138,-0.2586648) (0.0675395,-0.2592882) (0.0686759,-0.2599280) (0.0698234,-0.2605846) (0.0709824,-0.2612586) (0.0721531,-0.2619503)
(0.0733361,-0.2626605) (0.0745318,-0.2633896) (0.0757405,-0.2641384) (0.0769627,-0.2649073) (0.0781989,-0.2656972) (0.0794497,-0.2665087)
(0.0807154,-0.2673424) (0.0819968,-0.2681994) (0.0832942,-0.2690802) (0.0846085,-0.2699859) (0.0859400,-0.2709173) (0.0872896,-0.2718754)
(0.0886580,-0.2728612) (0.0900458,-0.2738758) (0.0914538,-0.2749204) (0.0928829,-0.2759960) (0.0943339,-0.2771041) (0.0958078,-0.2782460)
(0.0973054,-0.2794231) (0.0988279,-0.2806370) (0.1003762,-0.2818892) (0.1019517,-0.2831816) (0.1035554,-0.2845160) (0.1051887,-0.2858944)
(0.1068530,-0.2873189) (0.1085497,-0.2887918) (0.1102805,-0.2903156) (0.1120471,-0.2918928) (0.1138513,-0.2935263) (0.1156951,-0.2952191)
(0.1175805,-0.2969745) (0.1195099,-0.2987960) (0.1214856,-0.3006875) (0.1235105,-0.3026532) (0.1255873,-0.3046976) (0.1277192,-0.3068256)
(0.1299096,-0.3090426) (0.1321622,-0.3113546) (0.1344811,-0.3137679) (0.1368708,-0.3162897) (0.1393362,-0.3189279) (0.1418827,-0.3216910)
(0.1445163,-0.3245889) (0.1472437,-0.3276321) (0.1500724,-0.3308327) (0.1530105,-0.3342040) (0.1560676,-0.3377611) (0.1592540,-0.3415212)
(0.1625817,-0.3455036) (0.1660642,-0.3497304) (0.1697171,-0.3542269) (0.1735582,-0.3590224) (0.1776082,-0.3641508) (0.1818912,-0.3696518)
(0.1864355,-0.3755720) (0.1912747,-0.3819667) (0.1964487,-0.3889023) (0.2020058,-0.3964588) (0.2080051,-0.4047342) (0.2145192,-0.4138495)
(0.2216392,-0.4239566) (0.2294807,-0.4352488) (0.2381933,-0.4479765) (0.2479744,-0.4624703) (0.2590897,-0.4791764) (0.2719068,-0.4987124)
(0.2869478,-0.5219562) (0.3049764,-0.5501919) (0.3271388,-0.5853460) (0.3551766,-0.6303369) (0.3916320,-0.6893969) (0.4395173,-0.7674404)
(0.5000000,-0.8660254) (0.5687781,-0.9772587) (0.6382300,-1.0878367) (0.7035665,-1.1896660) (0.7636056,-1.2809603) (0.8187369,-1.3626071)
(0.8697100,-1.4360619) (0.9172373,-1.5026759) (0.9619077,-1.5635587) (1.0041899,-1.6195925) (1.0444542,-1.6714761) (1.0829948,-1.7197666)
(1.1200477,-1.7649128) (1.1558043,-1.8072808) (1.1904213,-1.8471723) (1.2240287,-1.8848389) (1.2567352,-1.9204927) (1.2886325,-1.9543137)
(1.3197985,-1.9864564) (1.3503002,-2.0170543) (1.3801952,-2.0462236) (1.4095337,-2.0740658) (1.4383593,-2.1006708) (1.4667106,-2.1261180)
(1.4946213,-2.1504782) (1.5221214,-2.1738147) (1.5492376,-2.1961843) (1.5759937,-2.2176386) (1.6024109,-2.2382239) (1.6285085,-2.2579827)
(1.6543035,-2.2769535) (1.6798116,-2.2951715) (1.7050469,-2.3126693) (1.7300220,-2.3294766) (1.7547486,-2.3456209) (1.7792373,-2.3611276)
(1.8034976,-2.3760203) (1.8275384,-2.3903209) (1.8513676,-2.4040499) (1.8749928,-2.4172261) (1.8984205,-2.4298674) (1.9216570,-2.4419905)
(1.9447080,-2.4536109) (1.9675787,-2.4647433) (1.9902739,-2.4754015) (2.0127980,-2.4855985) (2.0351550,-2.4953466) (2.0573488,-2.5046574)
(2.0793827,-2.5135419) (2.1012599,-2.5220104) (2.1229833,-2.5300729) (2.1445555,-2.5377388) (2.1659792,-2.5450169) (2.1872564,-2.5519158)
(2.2083893,-2.5584436) (2.2293798,-2.5646080) (2.2502295,-2.5704165) (2.2709402,-2.5758761) (2.2915133,-2.5809935) (2.3119499,-2.5857753)
(2.3322515,-2.5902277) (2.3524190,-2.5943567) (2.3724534,-2.5981680) (2.3923557,-2.6016673) (2.4121265,-2.6048597) (2.4317666,-2.6077504)
(2.4512766,-2.6103444) (2.4706570,-2.6126466) (2.4899084,-2.6146613) (2.5090310,-2.6163933) (2.5280252,-2.6178467) (2.5468913,-2.6190258)
(2.5656296,-2.6199347) (2.5842401,-2.6205771) (2.6027229,-2.6209571) (2.6210783,-2.6210783) (2.6393060,-2.6209442) (2.6574062,-2.6205584)
(2.6753787,-2.6199244) (2.6932234,-2.6190453) (2.7109403,-2.6179246) (2.7285290,-2.6165652) (2.7459895,-2.6149704) (2.7633214,-2.6131431)
(2.7805245,-2.6110863) (2.7975984,-2.6088027) (2.8145430,-2.6062954) (2.8313577,-2.6035669) (2.8480422,-2.6006199) (2.8645962,-2.5974572)
(2.8810192,-2.5940813) (2.8973107,-2.5904948) (2.9134704,-2.5867000) (2.9294976,-2.5826996) (2.9453920,-2.5784958) (2.9611530,-2.5740910)
(2.9767801,-2.5694876) (2.9922727,-2.5646878) (3.0076302,-2.5596939) (3.0228522,-2.5545081) (3.0379379,-2.5491326) (3.0528869,-2.5435695)
(3.0676984,-2.5378209) (3.0823720,-2.5318889) (3.0969068,-2.5257756) (3.1113024,-2.5194830) (3.1255579,-2.5130131) (3.1396729,-2.5063679)
(3.1536465,-2.4995494) (3.1674781,-2.4925595) (3.1811671,-2.4854001) (3.1947127,-2.4780732) (3.2081141,-2.4705805) (3.2213708,-2.4629240)
(3.2344820,-2.4551055) (3.2474469,-2.4471268) (3.2602648,-2.4389897) (3.2729350,-2.4306960) (3.2854567,-2.4222476) (3.2978292,-2.4136461)
(3.3100517,-2.4048934) (3.3221235,-2.3959910) (3.3340438,-2.3869409) (3.3458117,-2.3777446) (3.3574266,-2.3684040) (3.3688877,-2.3589206)
(3.3801941,-2.3492961) (3.3913451,-2.3395323) (3.4023399,-2.3296307) (3.4131776,-2.3195931) (3.4238576,-2.3094211) (3.4343789,-2.2991163)
(3.4447407,-2.2886803) (3.4549423,-2.2781147) (3.4649828,-2.2674212) (3.4748615,-2.2566014) (3.4845774,-2.2456569) (3.4941297,-2.2345892)
(3.5035177,-2.2233999) (3.5127404,-2.2120907) (3.5217970,-2.2006630) (3.5306868,-2.1891186) (3.5394088,-2.1774588) (3.5479622,-2.1656854)
(3.5563461,-2.1537997) (3.5645597,-2.1418035) (3.5726021,-2.1296983) (3.5804724,-2.1174855) (3.5881699,-2.1051668) (3.5956935,-2.0927436)
(3.6030425,-2.0802175) (3.6102159,-2.0675901) (3.6172128,-2.0548629) (3.6240325,-2.0420373) (3.6306738,-2.0291150) (3.6371361,-2.0160975)
(3.6434183,-2.0029862) (3.6495196,-1.9897827) (3.6554390,-1.9764885) (3.6611757,-1.9631052) (3.6667286,-1.9496342) (3.6720969,-1.9360771)
(3.6772796,-1.9224353) (3.6822758,-1.9087105) (3.6870846,-1.8949041) (3.6917049,-1.8810176) (3.6961358,-1.8670525) (3.7003764,-1.8530105)
(3.7044256,-1.8388928) (3.7082825,-1.8247012) (3.7119461,-1.8104371) (3.7154154,-1.7961020) (3.7186893,-1.7816974) (3.7217669,-1.7672249)
(3.7246472,-1.7526859) (3.7273290,-1.7380820) (3.7298113,-1.7234148) (3.7320931,-1.7086856) (3.7341733,-1.6938961) (3.7360509,-1.6790478)
(3.7377246,-1.6641422) (3.7391935,-1.6491808) (3.7404563,-1.6341652) (3.7415120,-1.6190968) (3.7423594,-1.6039773) (3.7429974,-1.5888081)
(3.7434247,-1.5735908) (3.7436401,-1.5583270) (3.7436425,-1.5430181) (3.7434306,-1.5276657) (3.7430032,-1.5122714) (3.7423589,-1.4968368)
(3.7414965,-1.4813633) (3.7404147,-1.4658525) (3.7391121,-1.4503060) (3.7375874,-1.4347254) (3.7358392,-1.4191122) (3.7338661,-1.4034680)
(3.7316665,-1.3877943) (3.7292392,-1.3720929) (3.7265825,-1.3563651) (3.7236950,-1.3406127) (3.7205751,-1.3248372) (3.7172212,-1.3090402)
(3.7136316,-1.2932233) (3.7098047,-1.2773882) (3.7057389,-1.2615364) (3.7014323,-1.2456696) (3.6968831,-1.2297893) (3.6920897,-1.2138973)
(3.6870499,-1.1979951) (3.6817621,-1.1820845) (3.6762241,-1.1661670) (3.6704339,-1.1502443) (3.6643895,-1.1343180) (3.6580887,-1.1183899)
(3.6515293,-1.1024617) (3.6447090,-1.0865349) (3.6376255,-1.0706114) (3.6302765,-1.0546928) (3.6226593,-1.0387808) (3.6147715,-1.0228772)
(3.6066104,-1.0069837) (3.5981732,-0.9911020) (3.5894572,-0.9752339) (3.5804595,-0.9593812) (3.5711769,-0.9435456) (3.5616064,-0.9277290)
(3.5517448,-0.9119331) (3.5415886,-0.8961597) (3.5311344,-0.8804107) (3.5203786,-0.8646879) (3.5093173,-0.8489932) (3.4979468,-0.8333285)
(3.4862629,-0.8176956) (3.4742614,-0.8020964) (3.4619379,-0.7865330) (3.4492878,-0.7710072) (3.4363063,-0.7555209) (3.4229883,-0.7400763)
(3.4093288,-0.7246752) (3.3953221,-0.7093197) (3.3809626,-0.6940119) (3.3662442,-0.6787539) (3.3511608,-0.6635477) (3.3357056,-0.6483955)
(3.3198718,-0.6332994) (3.3036521,-0.6182617) (3.2870389,-0.6032845) (3.2700240,-0.5883702) (3.2525991,-0.5735210) (3.2347550,-0.5587392)
(3.2164825,-0.5440273) (3.1977713,-0.5293876) (3.1786111,-0.5148227) (3.1589905,-0.5003349) (3.1388975,-0.4859270) (3.1183197,-0.4716015)
(3.0972434,-0.4573611) (3.0756543,-0.4432086) (3.0535372,-0.4291467) (3.0308756,-0.4151783) (3.0076521,-0.4013065) (2.9838478,-0.3875342)
(2.9594427,-0.3738646) (2.9344151,-0.3603009) (2.9087416,-0.3468464) (2.8823969,-0.3335047) (2.8553538,-0.3202792) (2.8275827,-0.3071738)
(2.7990513,-0.2941921) (2.7697246,-0.2813384) (2.7395642,-0.2686166) (2.7085282,-0.2560313) (2.6765702,-0.2435870) (2.6436394,-0.2312885)
(2.6096794,-0.2191409) (2.5746275,-0.2071496) (2.5384140,-0.1953203) (2.5009605,-0.1836591) (2.4621792,-0.1721723) (2.4219704,-0.1608671)
(2.3802210,-0.1497508) (2.3368016,-0.1388316) (2.2915631,-0.1281183) (2.2443328,-0.1176205) (2.1949087,-0.1073489) (2.1430532,-0.0973155)
(2.0884832,-0.0875334) (2.0308589,-0.0780177) (1.9697679,-0.0687858) (1.9047032,-0.0598577) (1.8350354,-0.0512572) (1.7599741,-0.0430128)
(1.6785211,-0.0351600) (1.5894208,-0.0277434) (1.4911470,-0.0208217) (1.3820642,-0.0144735) (1.2612092,-0.0088050) (1.1307199,-0.0039470)
(1.0000000,-0.0000000)}
\begin{minipage}[t]{0.48\textwidth}
\vspace{0pt}
\centering
\begin{tikzpicture}[x=0.79cm,y=0.79cm,>=Stealth,font=\small]
\path[use as bounding box] (-4.2,-3.0) rectangle (4.8,4.75);
\filldraw[fill=gray!10,draw=blue!55!black,line width=0.65pt]
 plot coordinates {\CubicBoundaryCoordinates} -- cycle;
\draw[gray,densely dashed,line width=0.45pt] (0.52,-0.16) rectangle (1.50,3.5);
\draw[red!65!black,->,line width=0.75pt] (1,0) -- (1.0833333,3);
\fill (1,0) circle (1.4pt) node[right=3pt] {$1$};
\fill (1.0833333,3) circle (1.6pt) node[right=3pt] {$a$};
\node at (-1.75,-1.30) {$\Omega_{3,36}$};
\node at (0.3,4.55) {The whole domain};
\end{tikzpicture}
\end{minipage}\hfill
\begin{minipage}[t]{0.49\textwidth}
\vspace{0pt}
\centering
\begin{tikzpicture}[x=1.60cm,y=1.60cm,>=Stealth,font=\small]
\path[use as bounding box] (-0.15,-0.43) rectangle (4.00,3.78);
\begin{scope}
\clip (0.52,-0.16) rectangle (1.50,3.5);
\filldraw[fill=gray!10,draw=blue!55!black,line width=0.75pt]
 plot coordinates {\CubicBoundaryCoordinates} -- cycle;
\end{scope}
\draw[gray,line width=0.35pt] (0.52,-0.16) rectangle (1.50,3.5);
\draw[red!65!black,line width=0.85pt] (1,0) -- (1.059218724,2.131874047);
\draw[red!65!black,densely dashed,line width=0.85pt]
 (1.059218724,2.131874047) -- (1.0833333,3);
\draw[red!65!black,->,line width=0.85pt] (1,0) -- (1.021,0.756);
\draw[gray,line width=0.5pt] (0.57,0.01194) -- (1.43,-0.01194);
\fill (1,0) circle (1.6pt) node[below=4pt] {$\zeta=1$};
\fill (1.059218724,2.131874047) circle (1.5pt);
\fill (1.0833333,3) circle (1.8pt);
\draw[gray,line width=0.4pt] (1.11,3) -- (1.70,3);
\node[anchor=west] at (1.72,3) {$a=\frac{13}{12}+3i$};
\draw[gray,line width=0.4pt] (1.06,1.08) -- (1.70,1.08);
\node[anchor=west,align=left] at (1.72,1.08)
 {outside $\Omega$\\along the normal};
\draw[gray,line width=0.4pt] (1.059218724,2.131874047) -- (1.70,2.131874047);
\node[anchor=west] at (1.72,2.131874047) {returns inside};
\node[blue!55!black] at (0.75,2.63) {$\Omega$};
\node at (1.82,3.67) {An enlarged view of the return};
\end{tikzpicture}
\end{minipage}
\endgroup
\caption{The cubic seed $|z^3+36i|<\sqrt{1297}$ of \cref{sec:cubic}, with its interior shaded.
The boxed region is enlarged on the right. The outward normal from $1$ crosses the exterior along the solid red segment. Along the dashed
segment, the ray has returned to the interior.}
\label{fig:cubic}
\end{figure}
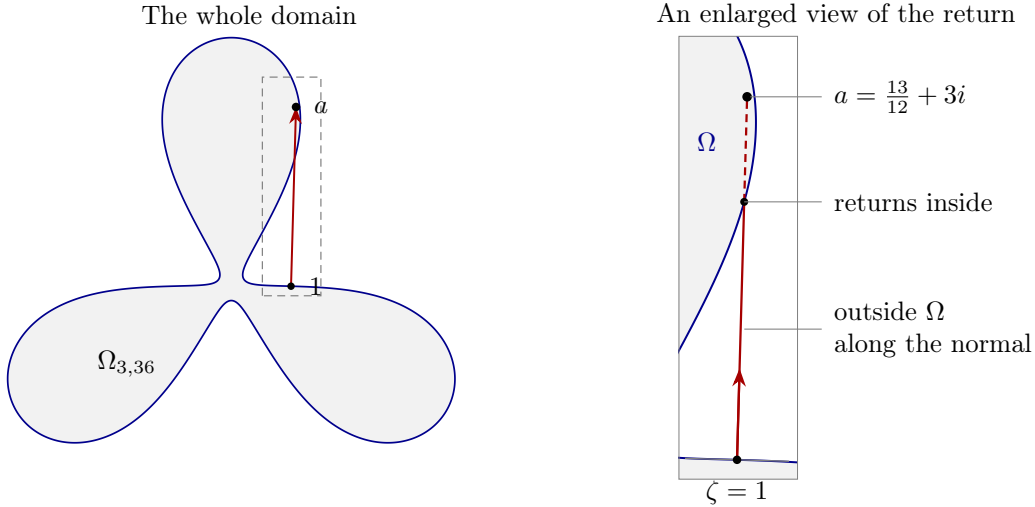

\Cref{sec:explicit-29984} gives a polynomial in $\Q(i)[z]$ of degree $29\,984$, with simple critical points and $\Smal(p,0)>1.00000049595801$, certified by interval arithmetic. By considering perturbations $\epsilon z^n$ for small $\epsilon$ and applying Rouch\'e's theorem, this example yields counterexamples in every degree $n\ge29\,984$.

\subsection{The mechanism}

Here is the algebra behind the construction. Let $K=\{z:|q(z)|\le1\}$ be a polynomial
lemniscate that is the closure of a Jordan domain, and let
$a\in\operatorname{int}K$. We first find a polynomial $f$ with
$f(a)=0$, $f'(a)=1$, $f'\ne0$ on $K$, and
\[
 \min_{z\in\partial K}\left|\frac{f(z)}{z-a}\right|>1.
\]
When \cref{thm:normal-main} applies, we construct a holomorphic model
and approximate it by $f$.

Now add a high power and consider the polynomial
\begin{equation}\label{eq:intro-power}
 P_N(z)=f(z)+\frac{q(z)^N-q(a)^N}{N}.
\end{equation}
On $K$, the added term is small. Outside, its derivative grows rapidly. For large $N$, a thin collar of $\partial K$ contains every critical point of $P_N$. At such a point $c$, where
$q'(c)\ne0$, the critical-point equation gives
\begin{equation}\label{eq:intro-identity}
 P_N(c)-f(c)
 =-\frac{q(c)f'(c)}{Nq'(c)}-\frac{q(a)^N}{N}.
\end{equation}
The error is therefore uniformly small at all critical points, including
those outside $K$.

Since every critical point lies near the boundary and the value error
is uniformly small, all the quotients $|P_N(c)/(c-a)|$ remain bounded
above $1$ for large $N$. Dividing by $|P_N'(a)|$, which tends to $1$,
preserves this strict inequality. Thus
$\Smal(P_N,a)>1$. For arbitrary analytic Jordan domains, we use Faber polynomials instead of the powers $q^N$;
see \cref{sec:faber}.

\subsection{Relation to the original construction}

The original construction found by GPT-6 Astra \cite{Adamczewski} uses a quartic lemniscate with
\[
 q(z)=\frac{z^4+64i}{\sqrt{4097}},\qquad
 a=\frac{33}{32}+2i.
\]
We make the following changes:
\begin{enumerate}
\item We show that the construction works whenever an outward normal
re-enters the domain and the curvature at its starting point satisfies
the inequality in \labelcref{eq:normal-return}. This is a weaker condition than the condition in the original construction.
\item We replace the primitive of a high power with the direct perturbation \labelcref{eq:intro-power}. At a critical point, \labelcref{eq:intro-identity} replaces the integral estimate by an algebraic expression for the value error. We work this out for a cubic lemniscate.
\item The discovery transcript of GPT-6 Astra already gives a transfer on analytic Jordan domains using primitives of Faber polynomials \cite{DiscoveryTranscript}. Here we use the Faber polynomials directly, which simplifies the critical-value estimate. We also locate every critical point to first order and determine the limiting distribution.
\end{enumerate}

\subsection{Outline} We assume familiarity with holomorphic functions, power series, Cauchy's formula, and the maximum modulus principle. Other tools are recalled when needed. Throughout,
\[
 \D=\{w\in\C:|w|<1\},\qquad \T=\partial\D.
\]
A conformal map is an injective holomorphic map. Its derivative never vanishes, and its inverse is holomorphic on the image. A critical point $c$ of a polynomial $p$ is \emph{simple} if $p''(c)\ne0$.

Background on the mean value problem and its dual is collected in \Cref{sec:previous}.
\Cref{sec:cubic} gives the cubic example; \Cref{sec:geometry} explains its geometry. \Cref{sec:faber} gives the transfer on arbitrary analytic Jordan domains and proves \cref{thm:normal-main}. \Cref{sec:sampling} locates the critical points to first order and determines their limiting distribution. The limit of the minimum in \cref{subsec:boundary-limit} then gives the boundary formulation of the optimal universal constant in \cref{prop:boundary-universal}.

\subsection{AI disclosure}\label{subsec:provenance}

Adamczewski designed and ran the experiment in which GPT-6 Astra
found the construction and its formal proof
\cite{Adamczewski}. The model was asked to prove or disprove the Lean statement of Smale's mean value conjecture with constant $1$ from Formal Conjectures \cite{FormalConjectures}, Google DeepMind's open-source library of formalized open problems. The successful attempt took about 15 hours and cost \$265 at the model's actual token prices.

Colbrook wrote the exposition in this paper after discussing it with Adamczewski and studying the experiment's transcript to understand the underlying mechanism. ChatGPT 6 was used for minor language editing of the penultimate draft in this stage only. This work also led to the explicit example in \cref{sec:explicit-29984}, which has a much smaller degree than the initial experiment suggested. That example was found by a parameter optimization run on Codex after human input on ways to lower the degree. The authors are responsible for the results, proofs, computations, accompanying code, and final manuscript.

\subsection*{Addendum}
While we were putting the finishing touches to this paper, the work of Jatar and Ng \cite{JatarNgBlaschke2026} appeared. They characterize the optimal universal constant $M$ through an extremal problem for univalent functions and finite Blaschke products, obtaining the lower bound $M\ge\sqrt{8081/6656}>1.1018$. Their work also gives the boundary characterization in \cref{prop:boundary-universal}, using Runge approximation and Hilbert's lemniscate theorem to pass from holomorphic models to polynomials. They further settle Problems~1A and~1B of
Smale's original paper \cite{Smale1981}, showing that the respective constants $4$ and $2$ cannot be reduced.

\section{Previous work on the mean value problem}\label{sec:previous}

Earlier results explain both the conjecture's plausibility and the role of geometry in attempts to prove it. Without loss of generality, we normalize polynomials by $p(0)=0$ and $p'(0)=1$. Define
\begin{equation}\label{eq:sharp-constants}
 M_n=\sup_{\deg p=n}\Smal(p,0),\qquad
 M=\sup_{n\ge2}M_n,
\end{equation}
where the supremum defining $M_n$ is over all normalised polynomials of degree $n$. Thus $M_n$ is sharp in degree $n$ and $M$ is sharp over all degrees. Smale conjectured that $M_n=(n-1)/n$, and hence that $M=1$.

\subsection{Newton's method and the origin of the constant 4}

Smale introduced the problem while studying the complexity of finding polynomial zeros by Newton's method \cite{Smale1981}. To continue a local inverse, one follows a solution of $p(z)=w$ as $w$ varies and critical values correspond to possible obstructions. Smale returned to the mean value problem in the addenda to the expanded version of his list of mathematical problems for the next century \cite[Addendum~1, pp.~21--22]{Smale1998}.

Smale's proof of the bound $4$, with crucial help acknowledged from Mike Shub, is geometric
\cite[Part~III, p.~33]{Smale1981}. Let $g$ be the inverse branch of $p$ with $g(0)=0$, defined on the largest disk $\{|w|<R\}$. Since $p(0)=0$ and $p'(0)=1$, we have $g'(0)=1$. By maximality, some critical value $p(c)$ must lie on $\{|w|=R\}$, with $c$ on the boundary of $g(\{|w|<R\})$. By the Koebe quarter theorem \cite{Duren}, the image of $g$ contains the disk of radius $R|g'(0)|/4=R/4$ centered at the origin. Hence $|c|\ge R/4$, and therefore
$
 |{p(c)}/{c}|\le 4.
$
To improve this bound, one must use the fact that the inverse comes from a polynomial.

\subsection{Degree-dependent improvements}

The following table lists bounds valid for every normalized polynomial in the stated degree range.\vspace{2mm}

\begin{center}
\renewcommand{\arraystretch}{1.5}
\begin{tabular}{@{}p{0.35\textwidth}p{0.44\textwidth}p{0.14\textwidth}@{}}
\hline
Authors & Upper bound for $M_n$ & Degrees\\
\hline
Beardon--Minda--Ng \cite{BeardonMindaNg}
 & $4^{(n-2)/(n-1)}$ & $n\ge2$\\
Conte--Fujikawa--Lakic \cite{ConteFujikawaLakic}
 & $4(n-1)/(n+1)$ & $n\ge2$\\
Fujikawa--Sugawa \cite{FujikawaSugawa}
 & $\displaystyle\frac{4\bigl(1+(n-2)4^{-1/(n-1)}\bigr)}{n+1}$ & $n\ge2$\\
Crane \cite{Crane2007}
 & $4-2.263/\sqrt n$ \quad(strict) & $n\ge8$\\
Dubinin \cite{Dubinin2006,DubininBlaschke2024}
 & $\tfrac{n-1}{n}\,4^{(n-2)/(n-1)}$ & $n\ge2$\\
\hline
\end{tabular}
\end{center}\vspace{2mm}

All five bounds are below $4$ and tend to $4$ as $n\to\infty$.
Beardon, Minda, and Ng based their bound on the hyperbolic metric \cite{BeardonMindaNg}. Coefficient estimates for injective holomorphic functions gave the bounds of Conte, Fujikawa, and Lakic \cite{ConteFujikawaLakic}, and Fujikawa and Sugawa combined the two approaches \cite{FujikawaSugawa}. Crane's stronger asymptotic improvement comes from logarithmic capacity and an extremal problem for connected sets containing prescribed points \cite{Crane2007}. Dubinin stated a further bound in 2006; his 2024 paper obtains it from finite Blaschke products \cite{Dubinin2006,DubininBlaschke2024}.

\subsection{Positive results under restrictions}

For some classes of polynomials, the proposed sharp bound is known to hold. Tischler proved several basic cases by studying critical points and critical values \cite{Tischler}. Two such classes are normalized polynomials with all zeros real, and those with all nonzero zeros of equal modulus. Hinkkanen and Kayumov discuss the earlier work of Palais on real zeros and of Sikorav in low degrees \cite{HinkkanenKayumovRays}. For background and further references, see Rahman and Schmeisser, Schmeisser's survey, and Sheil-Small \cite{RahmanSchmeisser,SchmeisserSurvey,SheilSmall}.

One can also restrict the critical points. Dubinin's bound is $(n-1)/n$ if all critical points have equal modulus \cite{Dubinin2006}. Hinkkanen and Kayumov's degree-independent bound is $2/3$ when all critical points are real; a cubic attains equality. They also studied configurations on two rays \cite{HinkkanenKayumov2010,HinkkanenKayumovRays}. Ng's upper bound $2$ applies to a class containing all odd polynomials with nonzero linear coefficient \cite{NgOdd}.

One can instead look for extremisers, or study a product of the quotients. Crane proved that extremisers exist and studied their geometry \cite{CraneExtremals}. Dimitrov related the product of the quotients to pairwise distances between zeros, with an electrostatic interpretation through logarithmic energy \cite{Dimitrov}. The minimum is at most the geometric mean.

\subsection{Stronger conjectures and numerical evidence}

Sendov and Marinov's computations, for degrees up to ten \cite{MarinovSendov}, give numerical evidence about candidate extremisers in low degree. 

Tischler proposed a stronger conjecture. His question was whether a normalised polynomial of degree $n$ always has a critical point $c$ satisfying
\begin{equation}\label{eq:tischler-disk}
 \left|\frac{p(c)}{c}-\frac12\right|\le\frac12-\frac1n.
\end{equation}
This restricts both the argument and the modulus, and the triangle inequality would then give $|p(c)/c|\le(n-1)/n$. Tyson's counterexamples to \labelcref{eq:tischler-disk} exist in every degree $n\ge5$. He also proved local positive results near a natural extremal family \cite{Tyson}. Failure of the disk condition does not by itself give $\Smal(p,0)>1$.

\subsection{Blaschke products and the dual problem}

There is a related mean value problem for holomorphic self-maps of the disk. For finite Blaschke products, Sheil-Small developed mean value inequalities \cite{SheilSmall}. Ng and Zhang proved degree-dependent estimates for both the mean value and dual problems in this setting; one dual estimate also transfers to polynomials \cite{NgZhang}. Dubinin's later work gives sharp bounds for the minimum modulus of critical values of Blaschke products, as well as estimates for divided differences \cite{DubininBlaschke2022,DubininBlaschke2024}.

The dual problem considers the largest critical-point quotient and was proposed independently by Dubinin and Sugawa and by Ng \cite{DubininSugawa,NgZhang}. Define
\begin{equation}\label{eq:dual-functional}
 \mathcal D(p)=\max_{p'(c)=0}\left|\frac{p(c)}{c}\right|.
\end{equation}
The conjecture is that $\mathcal D(p)\ge 1/n$ for every normalized polynomial of degree $n$. Equality is attained by
\[
 p(z)=\frac{(1+z)^n-1}{n},
\]
whose only critical point is $-1$.
The 2009 lower bound of Dubinin and Sugawa was $1/(n4^n)$ \cite{DubininSugawa}. With his finite-increment theorem, Dubinin improved it to $1/n^2$ \cite{Dubinin2010}. Ng and Zhang's different argument with finite Blaschke products gives a bound strictly greater than $4^{-n}$ \cite{NgZhang}. Hinkkanen, Kayumov, and Khammatova proved the conjecture in degrees five and six, and then seven \cite{HKK2019,HKK2025}. For odd polynomials, Tang's preprint uses symmetry to reduce a squared quotient to a polynomial critical-point problem and thereby proves the conjecture \cite{Tang2025}. A stronger inequality proposed by Kayumov and Khammatova implies the dual conjecture. They proved it when all critical points lie on one ray \cite{KayumovKhammatova2026}.

\subsection{Recent asymptotic results}
For certain polynomial sequences, Jatar and Ng obtain upper bounds approaching $1$ from Cunningham's refinement of the Koebe quarter theorem in terms of logarithmic capacity \cite{JatarNg,Cunningham}. Their hypotheses on the leading coefficient require an extra normalization. Their bound for the dual problem, which needs no such hypothesis, is
\begin{equation}\label{eq:recent-dual-bound}
 \mathcal D(p)\ge\frac{(n-\tfrac12)^{1/n}}{n^2}.
\end{equation}
The ingredients are Eremenko's derivative bounds for polynomial lemniscates and the deformation methods of Eremenko and Hayman \cite{Eremenko,EremenkoHayman,JatarNg}.

\section{A counterexample from a cubic seed}\label{sec:cubic}

We now consider a family of polynomial counterexamples generated by a cubic seed. This will explain the mechanism in a concrete setting. The idea is as follows. On a domain $\Omega$, choose an interior point $a$ and a holomorphic function $F$ satisfying $F(a)=0$, $F'(a)=1$, and $F'\ne0$ on $\ol\Omega$, whose divided-difference modulus exceeds $1$ on $\partial\Omega$. A sufficiently accurate polynomial approximation preserves these properties. Adding a sufficiently high power then pushes every critical point close to the boundary, and the critical-point equation keeps its values close to those of the approximation.

\subsection{The domain and its conformal coordinate}\label{subsec:cubic-coordinate}

Take
\begin{equation}\label{eq:cubic-data}
 q(z)=\frac{z^3+36i}{\sqrt{1297}},\qquad
 \Omega=\{z:|q(z)|<1\},\qquad
 a=\frac{13}{12}+3i.
\end{equation}
The domain is the inverse image under $z\mapsto z^3$ of a disk containing the origin, with $1\in\partial\Omega$. It is non-convex, allowing the outward normal return in \cref{fig:cubic}. We will use a conformal coordinate that maps $\Omega$ onto $\D$ to construct the holomorphic model.

\begin{proposition}[The cubic conformal coordinate]\label{prop:cubic-coordinates}
The domain $\Omega$ is an analytic Jordan domain. With the branch chosen so that $\Phi(1)=1$, the function
\begin{equation}\label{eq:cubic-phi}
 \Phi(w)=w(1+36i-36iw^3)^{-1/3}
\end{equation}
maps $\D$ conformally onto $\Omega$ and extends conformally across $\ol\D$. Moreover,
\[
 \Phi(1)=1,\qquad \Phi'(1)=1+36i,\qquad
 \frac{\Phi''(1)}{\Phi'(1)}=144i.
\]
\end{proposition}

\begin{proof}
We first check that the cube root can be chosen holomorphically across the closed disk and that $\Phi$ is injective. Let $D(w)=1+36i-36iw^3$. Its zeros have modulus $(\sqrt{1297}/36)^{1/3}>1$. A holomorphic logarithm of $D$ on a disk of radius greater than $1$ therefore defines the chosen cube root. If $\Phi(w_1)=\Phi(w_2)$, cubing gives
\[
 \frac{w_1^3}{1+36i-36iw_1^3}
 =\frac{w_2^3}{1+36i-36iw_2^3}.
\]
Hence $w_1^3=w_2^3$, the root factors agree, and $w_1=w_2$. Differentiation gives
\[
 \Phi'(w)=(1+36i)D(w)^{-4/3},
\]
which is nonzero and gives the stated boundary data.

The identity
\begin{equation}\label{eq:cubic-coordinate-modulus}
 |q(\Phi(w))|^2-1=\frac{|w|^6-1}{|D(w)|^2}
\end{equation}
shows that $\Phi(\D)\subset\Omega$, with boundary on $|q|=1$. The domain $\Omega$ is star-shaped about $0$: the disk $\{u:|u+36i|<\sqrt{1297}\}$ is convex and contains $0$, and $(tz)^3=t^3z^3$ for $0\le t\le1$. Thus $\Omega$ is connected. The image $\Phi(\D)$ is open in $\Omega$ and also relatively closed: a limit point in $\Omega$ cannot have preimages tending to $\T$, by \labelcref{eq:cubic-coordinate-modulus}. Thus, it is all of $\Omega$, and the extension of $\Phi$ gives the analytic Jordan boundary.
\end{proof}

The ratio $(\Phi(w)-a)/\Phi'(w)$ expresses the displacement from $a$ in the scale and direction of the conformal coordinate. Its value and derivative at $w=1$ will determine the motion of the critical points of the holomorphic model.

\begin{corollary}[The interior point and boundary data]\label{cor:cubic}
The point $a$ in \labelcref{eq:cubic-data} belongs to $\Omega$. Define
\begin{equation}\label{eq:J-definition}
 J(w)=\frac{\Phi(w)-a}{\Phi'(w)}.
\end{equation}
Then $J(1)=-1/12$ and $J'(1)=1+12i$.
\end{corollary}

\begin{proof}
We have
\[
 a^3+36i=-\frac{48347}{1728}+\frac{313}{16}i.
\]
The real and imaginary parts have absolute values strictly less than $28$ and $20$, respectively. Therefore
\[
 |a^3+36i|^2<28^2+20^2=1184<1297.
\]
The two values of $J$ and $J'$ follow from \cref{prop:cubic-coordinates} and $a=\Phi(1)+\Phi'(1)/12$.
\end{proof}

Since $\Phi'(1)$ points along the outward normal and
\[
 a-1=\frac{\Phi'(1)}{12},
\]
the outward normal ray returns to the interior point $a$ in $\Omega$. The negative real value $J(1)=-1/12$ makes the derivative zeros move tangentially to leading order.

\subsection{A holomorphic model with a strict boundary margin}\label{subsec:cubic-model}

We use a Blaschke factor to obtain the required boundary margin. Let $W=\Phi^{-1}$, $w_a=W(a)$, and consider
\[
 B_r(w)=\frac{w-r}{1-rw},\qquad 0<r<1.
\]
This function has one zero in $\D$, and the identity
\begin{equation}\label{eq:blaschke-modulus}
 1-|B_r(w)|^2
 =\frac{(1-r^2)(1-|w|^2)}{|1-rw|^2}
\end{equation}
implies that $|B_r|=1$ on $\T$ and $|B_r|<1$ in $\D$. Thus $B_r(W(z))$ has constant modulus on $\partial\Omega$ and a smaller modulus at $a$. Dividing by its value at $a$ gives the strict boundary inequality and the required normalization. We define the function
\begin{equation}\label{eq:blaschke-model}
 F_r(z)=(z-a)\frac{B_r(W(z))}{B_r(w_a)}.
\end{equation}
Whenever $B_r(w_a)\ne0$, we have
\[
 F_r(a)=0,\qquad F_r'(a)=1,\qquad
 \left|\frac{F_r(z)}{z-a}\right|
 =\Lambda_r:=\frac1{|B_r(w_a)|}>1
 \quad(z\in\partial\Omega).
\]
For $r$ sufficiently close to $1$, the denominator is nonzero because $B_r(w_a)\to-1$. The function has two distinct zeros, $a$ and $\Phi(r)$, and approaches $z-a$ away from the boundary point $1$.

Since an interior critical point could survive the later perturbation, we must also show that $F_r'\ne0$ on $\ol\Omega$. Using $J$ from \cref{cor:cubic} and $W'(\Phi(w))=1/\Phi'(w)$, differentiation gives
\begin{equation}
 F_r'(\Phi(w))
 =\frac{B_r(w)+J(w)B_r'(w)}{B_r(w_a)}
 =\frac{(w-r)(1-rw)+(1-r^2)J(w)}
         {B_r(w_a)(1-rw)^2}.\label{eq:model-derivative}
\end{equation}
The critical points in $\ol\Omega$ correspond to zeros in $\ol\D$ of
\[
 N_r(w)=(w-r)(1-rw)+(1-r^2)J(w).
\]
As $r\uparrow1$, this numerator converges to $-(w-1)^2$ and is eventually nonzero outside any fixed neighborhood of $w=1$ in $\ol\D$. The remaining task is to locate the two zeros near $1$.

Set $r=1-\varepsilon$. To leading order, the equation for these two zeros is
\[
 -(w-1)^2+2\varepsilon J(1)=0.
\]
For the cubic, $J(1)=-1/12$, so these displacements are purely imaginary: $w-1\approx\pm i\sqrt{\varepsilon/6}$. These are tangential to the unit circle at $1$. The next term in the expansion decides on which side of the unit circle these zeros lie.

\begin{proposition}[The cubic holomorphic model]\label{prop:cubic-model}
For every $r<1$ sufficiently close to $1$, the function $F_r$ in \labelcref{eq:blaschke-model} is holomorphic on a neighbourhood of $\ol\Omega$ and satisfies
\[
 F_r(a)=0,\qquad F_r'(a)=1,\qquad F_r'\ne0\quad\text{on }\ol\Omega,
 \qquad
 \left|\frac{F_r(z)}{z-a}\right|=\Lambda_r>1
 \quad(z\in\partial\Omega).
\]
\end{proposition}

\begin{proof}
The Blaschke identity gives the normalizations and boundary modulus. Since the pole $1/r$ lies outside $\ol\D$ and $W$ extends conformally across $\partial\Omega$, $F_r$ is holomorphic on a neighbourhood of $\ol\Omega$. We use the data in \cref{cor:cubic} to exclude zeros of $N_r$ in $\ol\D$.

The leading equation in the discussion preceding the proposition shows that $w-1$ is of size $\sqrt\varepsilon$. Set $\varepsilon=t^2$ and $w=1+tu$. The exact identity
\begin{equation}\label{eq:quadratic-product}
 (w-r)(1-rw)=-x^2+\varepsilon x^2+\varepsilon^2+\varepsilon^2x,
 \qquad x=w-1,
\end{equation}
gives
\[
 t^{-2}N_{1-t^2}(1+tu)
 =-u^2-\frac16+2(1+12i)tu+O(t^2).
\]
The scaled function extends holomorphically to $t=0$. At that point its roots $u_0=\pm i/\sqrt6$ are simple, so the holomorphic implicit function theorem continues each as a holomorphic function of $t$. If $u(t)=u_0+u_1t+O(t^2)$, equating coefficients of $t$ gives
\[
 -2u_0u_1+2(1+12i)u_0=0,
\]
so $u_1=1+12i$. Hence
\begin{equation}\label{eq:cubic-root-expansion}
 w_\pm=1\pm i\sqrt{\varepsilon/6}
 +(1+12i)\varepsilon+O(\varepsilon^{3/2}),
\end{equation}
and
\begin{equation}\label{eq:cubic-root-modulus}
 |w_\pm|^2=1+\frac{13}{6}\varepsilon+O(\varepsilon^{3/2}).
\end{equation}
Both roots are outside $\ol\D$ for small positive $\varepsilon$.

On the boundary of a small fixed disk about $1$, uniform convergence $N_{1-\varepsilon}\to-(w-1)^2$ allows us to apply Rouch\'e's theorem. There are exactly two zeros inside, counted with multiplicity. On the compact remainder of $\ol\D$, the limiting function is bounded away from zero, so it has no further zeros. Since both displayed roots lie outside $\ol\D$, the derivative is nonzero on $\ol\Omega$.
\end{proof}

The coefficient $13/6$ arises from two contributions. For $x=w_\pm-1$,
\[
 |1+x|^2-1=2\Ree x+|x|^2.
\]
The next displacement in \labelcref{eq:cubic-root-expansion} contributes $2\varepsilon$ through $2\Ree x$, while the leading tangential displacement contributes $\varepsilon/6$ through $|x|^2$. Thus tangential motion of size $\sqrt\varepsilon$ affects the radius at the same order as the correction, giving $2+1/6=13/6$. In \cref{sec:geometry}, this calculation yields a condition on the outward-normal return and boundary curvature.

\subsection{Preserving the strict inequalities under approximation}\label{subsec:approximation}

Fix an admissible $r$. Accurate enough approximation of $F_r$ and $F_r'$ preserves the positive margins $\min_{\ol\Omega}|F_r'|$ and $\Lambda_r-1$. The next lemma shows that they also preserve the normalization at $a$. We prove the result for general domains since it will be used later.

\begin{samepage}
\begin{lemma}[Normalised polynomial approximation]\label{lem:approximation}
Let $\Omega$ be an analytic Jordan domain, $a\in\Omega$, and let $F$ be holomorphic on a neighborhood of $\ol\Omega$. Suppose
\[
 F(a)=0,\qquad F'(a)=1,\qquad F'\ne0\ \text{on }\ol\Omega.
\]
For every $\eta>0$, there is a polynomial $f$ satisfying
\[
 f(a)=0,\qquad f'(a)=1,\qquad f'\ne0\ \text{on }\ol\Omega
\]
and
\begin{equation}\label{eq:approximation-quotient}
 \max_{z\in\partial\Omega}
 \left|\frac{f(z)-F(z)}{z-a}\right|<\eta.
\end{equation}
In particular, a strict boundary minimum greater than $1$ persists.
\end{lemma}
\end{samepage}

\begin{proof}
Choose a slightly larger Jordan compact set $K_+$ whose interior contains $\ol\Omega$ and on a neighborhood of which $F$ is holomorphic. Extend the domain slightly across its analytic boundary to obtain such a set. Its complement is connected. Runge's theorem therefore gives polynomials $h_j$ converging uniformly to $F$ on $K_+$.

For a compact set $E$, define $\|h\|_E=\sup_{z\in E}|h(z)|$. Choose $\rho>0$ so that every closed disk of radius $\rho$ centered in $\ol\Omega$ lies in $K_+$. Cauchy's derivative formula gives
\[
 \|h_j'-F'\|_{\ol\Omega}
 \le\frac1\rho\|h_j-F\|_{K_+}\rightarrow0.
\]
For large $j$, define
\[
 f_j(z)=\frac{h_j(z)-h_j(a)}{h_j'(a)}.
\]
Then $f_j(a)=0$, $f_j'(a)=1$, and $f_j\to F$, $f_j'\to F'$ uniformly on $\ol\Omega$. The bounds $\min_{\ol\Omega}|F'|>0$ and $\dist(a,\partial\Omega)>0$ give the derivative condition and the quotient estimate for sufficiently large $j$.
\end{proof}

\subsection{Bringing every critical point to the boundary}\label{subsec:cubic-power}

We now return to the domain in \cref{eq:cubic-data}.
Let $f$ be a polynomial satisfying the properties in \cref{lem:approximation}, and set $K=\ol\Omega$. Consider
\[
 P_N(z)=f(z)+\frac{q(z)^N-q(a)^N}{N}.
\]
Subtracting $q(a)^N$ ensures that $P_N(a)=0$ and the factor $1/N$ makes the perturbation small on $K$. Differentiation cancels this factor and gives
\begin{equation}\label{eq:power-derivative}
 P_N'(z)=f'(z)+q(z)^{N-1}q'(z).
\end{equation}
The added derivative term decays exponentially on every compact subset of $\Omega$. Its modulus grows when $|q|>1$ and $q'\ne0$. We therefore expect cancellation with $f'$ near $|q|=1$. To exclude distant critical points, we count all derivative zeros inside a fixed outer level.

A \emph{collar} is a thin open band around a Jordan curve. The theorem below confines every critical point to a collar of width $O(N^{-1})$ around the curve.

\begin{theorem}[Power transfer for the cubic]\label{thm:cubic-transfer}
Let $q,a,\Omega$ be as in \labelcref{eq:cubic-data}, and let $K=\ol\Omega$. Suppose the polynomial $f$ satisfies
\begin{equation}\label{eq:cubic-base-normalisation}
 f(a)=0,\qquad f'(a)=1,\qquad f'(z)\ne0\quad(z\in K).
\end{equation}
Set
\[
 L=\min_{z\in\partial K}\left|\frac{f(z)}{z-a}\right|.
\]
For all sufficiently large $N$, the polynomial $P_N$ has degree $3N$, all its critical points are simple, and
\begin{equation}\label{eq:cubic-power-distance}
 \max_{P_N'(c)=0}\dist(c,\partial K)=O(N^{-1}).
\end{equation}
The identity \labelcref{eq:intro-identity} holds at every critical point. In particular,
\begin{equation}\label{eq:cubic-power-lower}
 \Smal(P_N,a)\ge L-O(N^{-1}).
\end{equation}
All constants depend only on the fixed data $q,f,a$. If $L>1$, then $\Smal(P_N,a)>1$ for all sufficiently large $N$.
\end{theorem}

\begin{proof}
We first prove \cref{eq:cubic-power-distance}. For every fixed $\rho<1$, the second term in \labelcref{eq:power-derivative} tends uniformly to zero on the inner level $\{|q|\le\rho\}$. Since $\min_K|f'|>0$, the derivative $P_N'$ is eventually nonzero there. To deal with a fixed outer level, note that $q'(z)=3z^2/\sqrt{1297}$ has its only zero at $0$, strictly inside $K$, and \cref{prop:cubic-coordinates} gives the analytic boundary. The levels $\{|q|=R\}$ for $R>1$ sufficiently close to $1$ are analytic Jordan curves surrounding $K$. On such a level,
\[
 |q^{N-1}q'|=R^{N-1}|q'|>|f'|
\]
for all sufficiently large $N$. Rouch\'e's theorem with $g=q^{N-1}q'$ and $h=f'$ gives the same zero count for $P_N'$ and $q^{N-1}q'$ in the enclosed domain. All zeros of $q$ and $q'$ lie there, so the number is
$
 3(N-1)+2=3N-1.
$
Once $3N>\deg f$, this is the full degree of $P_N'$. Thus, there are no critical points elsewhere in the plane.

The inner and outer levels can be chosen arbitrarily close to $\partial K$. Fix them so that the resulting closed collar $A$ contains neither $a$ nor a zero of $q$, $q'$, or $f'$. Every critical point lies in $A$ for all sufficiently large $N$. At a critical point $c$, equation~\labelcref{eq:power-derivative} implies
\begin{equation}\label{eq:power-log}
 (N-1)\log|q(c)|=\log\left|\frac{f'(c)}{q'(c)}\right|.
\end{equation}
The right side is uniformly bounded on $A$. For a distance estimate, set $u=\log|q|$. At a boundary point $\zeta$, the outward normal derivative is $|q'(\zeta)|>0$. Taylor's formula therefore gives
\[
 u(\zeta+t\cdot\nu(\zeta))=t|q'(\zeta)|+O(t^2),
\]
uniformly on the boundary. These short normal segments cover a sufficiently narrow collar. Since $|q'|$ has a positive boundary minimum, we obtain
\[
 \dist(z,\partial K)\le C|\log|q(z)||\qquad(z\in A).
\]
This proves \labelcref{eq:cubic-power-distance}.

Multiplying \labelcref{eq:power-derivative} at $c$ by $q(c)/q'(c)$ yields
\[
 q(c)^N=-\frac{q(c)f'(c)}{q'(c)},
\]
so the high power is uniformly bounded at all critical points. Substitution into $P_N$ proves \labelcref{eq:intro-identity}. Also
\begin{equation}\label{eq:power-at-a}
 P_N(a)=0,\qquad
 P_N'(a)=1+q(a)^{N-1}q'(a)=1+O(\sigma^N)
\end{equation}
for some $\sigma\in(0,1)$. The critical-value identity gives $P_N(c)=f(c)+O(N^{-1})$, uniformly over all critical points. Since the collar stays away from $a$, division by $c-a$ preserves this estimate. Consequently,
\[
 \left|\frac{P_N(c)}{(c-a)P_N'(a)}\right|
 =\left|\frac{f(c)}{c-a}\right|+O(N^{-1})
 \ge L-O(N^{-1}).
\]
For the inequality, compare $c$ with a nearest boundary point using \labelcref{eq:cubic-power-distance}: the function $z\mapsto|f(z)/(z-a)|$ is Lipschitz on a neighborhood of the collar. Taking the minimum over the critical points proves \labelcref{eq:cubic-power-lower}.

At a critical point, substitution of the derivative equation into the second derivative gives
\begin{equation}\label{eq:power-second}
 P_N''(c)=f''(c)
 -(N-1)\frac{f'(c)q'(c)}{q(c)}
 -\frac{f'(c)q''(c)}{q'(c)}.
\end{equation}
The middle term has modulus bounded below by a positive constant times $N$, and the other two terms are uniformly bounded. Thus $P_N''(c)\ne0$ for large $N$.
\end{proof}

We need $f'\ne0$ on $K$. An interior zero of $f'$ would survive, since the derivative perturbation is exponentially small on a small surrounding circle.

So far, the modulus of the critical-point equation has controlled distance from the boundary. In \cref{sec:sampling}, we use its argument to locate the critical points along the boundary. There we show that every boundary point is approached and determine the exact limiting boundary minimum.

\subsection{Completing the construction}\label{subsec:cubic-completion}

We now choose the approximation and the high power to retain a fixed part of the boundary margin.

\begin{corollary}[A family from the cubic seed]\label{cor:cubic-polynomials}
Let $q$ and $a$ be as in \labelcref{eq:cubic-data}.
There are a fixed polynomial $f$ with $f(a)=0$, $f'(a)=1$, and constants $\delta>0$ and $N_0$ such that, for every $N\ge N_0$, the polynomial
\begin{equation}\label{eq:cubic-witness}
 p_N(z)=\frac{\displaystyle f(a+z)+
                  \frac{q(a+z)^N-q(a)^N}{N}}
                 {\displaystyle 1+q(a)^{N-1}q'(a)}
\end{equation}
has degree $3N$ and satisfies
\[
 p_N(0)=0,\qquad p_N'(0)=1,\qquad \Smal(p_N,0)\ge1+\delta.
\]
Every critical point is simple.
\end{corollary}

\begin{proof}
Fix $r$ as in \cref{prop:cubic-model} and set $g=\Lambda_r-1>0$. The approximation in \cref{lem:approximation}, with $\eta=g/4$, gives a fixed polynomial $f$ satisfying \labelcref{eq:cubic-base-normalisation} and $L>1+3g/4$.

Here $f$ is fixed, so \cref{thm:cubic-transfer} gives $\Smal(P_N,a)\ge L-C/N$, where $C$ does not depend on $N$. Choose $N_0$ large enough that, for every $N\ge N_0$, the bound $C/N\le g/4$ holds, $P_N$ has degree $3N$ and simple critical points, and $P_N'(a)\ne0$. Then $\Smal(P_N,a)>1+g/2$.

The normalisation $p_N(z)=P_N(a+z)/P_N'(a)$ gives $p_N(0)=0$ and $p_N'(0)=1$, and preserves the mean-value functional. Set $\delta=g/2$ to finish the proof.
\end{proof}

\section{The geometry behind the construction}\label{sec:geometry}

Two boundary data govern the cubic root calculation in the previous section. The value of $J$ at $1$ sets the leading direction and its derivative supplies the next term. They give a strict escape criterion for the pair of zeros. 

We now generalize this criterion. Let $\Omega$ be an analytic Jordan domain and $a\in\Omega$.
On $\partial\Omega$, the criterion becomes an outward normal return together with a curvature inequality.
The Riemann mapping theorem supplies a conformal map $\Phi:\D\to\Omega$, and the analytic boundary allows conformal extensions of $\Phi$ and its inverse across their respective boundaries. Set $W=\Phi^{-1}$ and $w_a=W(a)$, and define $F_r$ and $J$ by \labelcref{eq:blaschke-model,eq:J-definition}. The Blaschke identity gives the normalizations and strict boundary inequality. To exclude derivative zeros, consider the numerator
\[
 N_r(w)=(w-r)(1-rw)+(1-r^2)J(w).
\]
As before, as $r\uparrow1$, this converges uniformly to $-(w-1)^2$, so only the two zeros near $1$ need further study.

\subsection{Tangential splitting and escape from the disk}\label{subsec:general-splitting}

Set $r=1-\varepsilon$, with $\varepsilon>0$ small. The leading root equation is $(w-1)^2=2\varepsilon J(1)$. We have $J(1)\ne0$, since $a\in\Omega$ and $\Phi(1)\in\partial\Omega$. Unless $J(1)$ is negative real, the two square roots have real parts of opposite signs, and one zero enters $\D$. With $J(1)=-\tau<0$, the leading motion is tangential, as in the cubic. The next term must then be compared with the curvature of the circle.

\begin{lemma}[Tangential splitting]\label{lem:splitting}
Let $J$ be holomorphic on a neighbourhood of $\ol\D$, and suppose that $J(1)=-\tau<0$.
Set $b=J'(1)$ and
\[
 N_r(w)=(w-r)(1-rw)+(1-r^2)J(w).
\]
As $\varepsilon\downarrow0$, the two zeros of $N_{1-\varepsilon}$ near $1$, counted with multiplicity, are
\begin{equation}\label{eq:root-expansion}
 w_\pm=1\pm i\sqrt{2\tau\varepsilon}
             +b\varepsilon+O(\varepsilon^{3/2}),
\end{equation}
and they are simple for small $\varepsilon>0$. In particular,
\begin{equation}\label{eq:root-modulus}
 |w_\pm|^2
 =1+2(\tau+\Ree b)\varepsilon+O(\varepsilon^{3/2}).
\end{equation}
If $\tau+\Ree b>0$, then $N_{1-\varepsilon}$ has no zeros in $\ol\D$ for all sufficiently small $\varepsilon>0$. If $\tau+\Ree b<0$, both displayed zeros lie in $\D$ for all sufficiently small $\varepsilon>0$.
\end{lemma}

\begin{proof}
We rescale as in the proof of \cref{prop:cubic-model} with $\varepsilon=t^2$ and $w=1+tu$. By \labelcref{eq:quadratic-product},
\[
 t^{-2}N_{1-t^2}(1+tu)
 =-u^2-2\tau+2btu+O(t^2).
\]
At $t=0$ the two roots $u_0=\pm i\sqrt{2\tau}$ are simple. The implicit-function theorem gives $u(t)=u_0+bt+O(t^2)$, proving \labelcref{eq:root-expansion}. Taking squared moduli then gives \labelcref{eq:root-modulus}. Rouch\'e's theorem on a fixed small circle about $1$ counts these two zeros. Uniform convergence to $-(w-1)^2$ excludes zeros on the compact remainder of $\ol\D$. Thus, in either strict case, the roots lie on the side of $\T$ indicated by the sign of $\tau+\Ree b$.
\end{proof}

For $x=w-1$, the unit circle has equation
\[
 2\Ree x+|x|^2=0.
\]
The tangential term in \labelcref{eq:root-expansion} contributes $2\tau\varepsilon$ through $|x|^2$, while the next displacement contributes $2\Ree b\,\varepsilon$. Hence
\begin{equation}\label{eq:root-distance}
 |w_\pm|-1=(\tau+\Ree b)\varepsilon+O(\varepsilon^{3/2}).
\end{equation}
Even an inward correction, $\Ree b<0$, can leave both roots outside, because the circle curves away from its tangent at the same order; see \cref{fig:splitting}. For the cubic construction in \cref{sec:cubic}, $\tau=1/12$ and $b=1+12i$, so $2(\tau+\Ree b)=13/6$, recovering \labelcref{eq:cubic-root-modulus}.

\begin{figure}[t]
\centering
\begin{minipage}[t]{0.47\textwidth}
\vspace{0pt}
\centering
\begin{tikzpicture}[x=1.85cm,y=1.85cm,>=Stealth,font=\small]
\path[use as bounding box] (-1.25,-1.3) rectangle (2.35,1.3);
\fill[gray!10] (0,0) circle (1);
\draw[blue!55!black,line width=0.7pt] (0,0) circle (1);
\draw[gray,densely dashed] (1,-0.9) -- (1,0.92);
\node[anchor=south] at (1,0.93) {tangent};
\fill (0,0) circle (1pt) node[below left] {$0$};
\node at (-0.38,0.4) {$\mathbb D$};
\draw[->,red!65!black,line width=0.8pt]
  (1,0.025) -- (1,0.408);
\draw[->,red!65!black,line width=0.8pt]
  (1,-0.025) -- (1,-0.408);
\fill (1,0) circle (1.5pt) node[right=3pt] {$1$};
\fill[red!65!black] (0.957225,0.4221) circle (1.6pt);
\fill[red!65!black] (0.957225,-0.4221) circle (1.6pt);
\node[anchor=west] at (1.09,0.44) {$w_+$};
\node[anchor=west] at (1.09,-0.44) {$w_-$};
\node[anchor=west,align=left] at (1.14,0.2)
 {$O(\sqrt\varepsilon)$};
\node at (0.5,-1.23) {(a) The leading displacement is tangential.};
\end{tikzpicture}
\end{minipage}\hfill
\begin{minipage}[t]{0.51\textwidth}
\vspace{0pt}
\centering
\begin{tikzpicture}[x=12cm,y=12cm,>=Stealth,font=\small]
\path[use as bounding box] (0.79,0.17) rectangle (1.40,0.571);
\begin{scope}
\clip (0.805,0.262) rectangle (1.147,0.548);
\fill[gray!10] (0,0) circle (1);
\draw[blue!55!black,line width=0.7pt] (0,0) circle (1);
\draw[gray,densely dashed] (1,0.25) -- (1,0.56);
\end{scope}
\draw[gray,densely dotted] (0.9006,0.424264) -- (1.13,0.424264);
\draw[red!65!black,->,line width=0.9pt]
 (0.998,0.424264) -- (0.962,0.4223);
\draw[fill=white,line width=0.7pt] (1,0.424264) circle (1.5pt);
\fill[red!65!black] (0.957225,0.4221) circle (1.8pt);
\node[anchor=south east] at (0.95,0.428) {$w_+$};
\node[anchor=west,align=left] at (1.023,0.456)
 {$1+i\sqrt{2\tau\varepsilon}$};
\draw[gray,line width=0.4pt] (1.015,0.45) -- (1.002,0.429);
\draw[gray,line width=0.4pt] (1.02,0.375) -- (0.978,0.414);
\node[anchor=west,align=left] at (1.026,0.362)
 {inward correction\\of order $\varepsilon$};
\node[anchor=west] at (1.018,0.528) {$\operatorname{Re}w=1$};
\node[blue!55!black,anchor=east] at (0.838,0.53) {$\mathbb T$};
\node[anchor=west] at (0.806,0.29) {$\mathbb D$};
\node at (1.075,0.189) {(b) Curvature leaves room for the correction.};
\end{tikzpicture}
\end{minipage}
\caption{Tangential escape from the unit disk. The leading
displacements have size $\sqrt\varepsilon$. An inward correction of size $\varepsilon$ can still leave
the root outside because the circle curves away from its tangent at the same
order. The condition $\tau+\operatorname{Re}b>0$ compares these two
effects.}
\label{fig:splitting}
\end{figure}
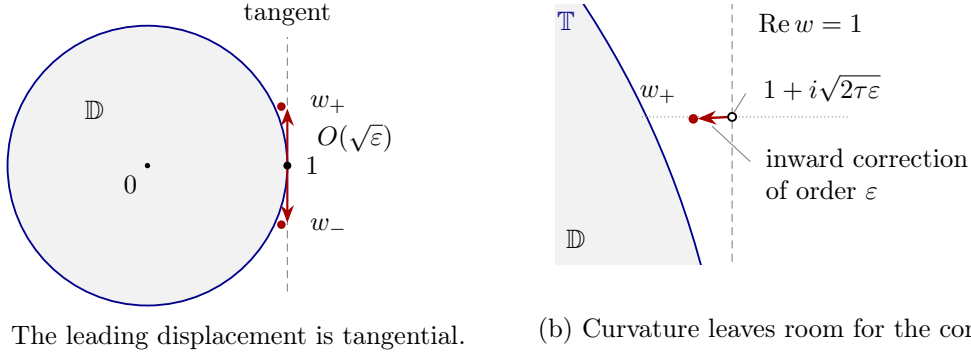

\begin{samepage}
\begin{proposition}[A holomorphic model with nonvanishing derivative]\label{prop:model}
Let $\Phi:\D\to\Omega$ be a conformal map onto $\Omega$ that extends conformally to a neighborhood of $\ol\D$, let $a\in\Omega$, and define $J$ by \labelcref{eq:J-definition}. Suppose
\begin{equation}\label{eq:J-conditions}
 J(1)=-\tau<0,\qquad \tau+\Ree J'(1)>0.
\end{equation}
Then, for every $r<1$ sufficiently close to $1$, the function $F_r$ in \labelcref{eq:blaschke-model} is holomorphic on a neighbourhood of $\ol\Omega$ and satisfies
\begin{equation}\label{eq:model-properties}
 \begin{gathered}
 F_r(a)=0,\qquad F_r'(a)=1,\qquad F_r'\ne0\quad\hbox{on }\ol\Omega,\\
 \left|\frac{F_r(z)}{z-a}\right|=\Lambda_r>1
 \quad(z\in\partial\Omega),\qquad
 \Lambda_r=|B_r(w_a)|^{-1}.
 \end{gathered}
\end{equation}
\end{proposition}
\end{samepage}

\begin{proof}
For each fixed $r$ sufficiently close to $1$, $B_r(w_a)\ne0$ and the pole $1/r$ lies outside $\ol\D$. The conformal extension of $W$ across $\partial\Omega$ makes $F_r$ holomorphic on a neighborhood of $\ol\Omega$. The Blaschke identity gives the normalizations and boundary modulus. By \cref{lem:splitting}, the numerator in \labelcref{eq:model-derivative} has no zeros in $\ol\D$, so $F_r'\ne0$ on $\ol\Omega$.
\end{proof}

\subsection{The normal-return and curvature conditions}\label{subsec:normal-geometry}

We now translate \labelcref{eq:J-conditions} back to the boundary of $\Omega$. The calculation below gives the criterion in \cref{thm:normal-main}.

\begin{proposition}[Intrinsic form of the boundary condition]\label{prop:curvature}
Let $\zeta=\Phi(1)$. The conditions in \labelcref{eq:J-conditions} are equivalent to
\[
 a=\zeta+d\cdot\nu(\zeta),\qquad d>0,\qquad
 1+d\cdot\kappa(\zeta)>0,
\]
with $\tau=d/|\Phi'(1)|$.
\end{proposition}

\begin{proof}
The tangent to $t\mapsto\Phi(e^{it})$ at $t=0$ is $i\Phi'(1)$, so the outward unit normal is
\[
 \nu(\zeta)=\frac{\Phi'(1)}{|\Phi'(1)|}.
\]
Thus $J(1)=-\tau$ means $a=\Phi(1)+\tau\Phi'(1)$. Differentiating $J$ gives
\[
 J'(1)=1+\tau\frac{\Phi''(1)}{\Phi'(1)}.
\]
For a regular counterclockwise parametrisation $\gamma(t)$, signed curvature is
\[
 \kappa=\frac1{|\gamma'|}\frac{\mathrm d}{\mathrm dt}\arg\gamma'
 =\frac{\Imm(\gamma''/\gamma')}{|\gamma'|}.
\]
For $\gamma(t)=\Phi(e^{it})$, we have
\[
 \gamma'(0)=i\Phi'(1),\qquad
 \gamma''(0)=-\Phi'(1)-\Phi''(1).
\]
These formulas give
\begin{equation}\label{eq:curvature-formula}
 \kappa(\zeta)
 =\frac{\Ree\bigl(1+\Phi''(1)/\Phi'(1)\bigr)}{|\Phi'(1)|}.
\end{equation}
Consequently
\begin{equation}\label{eq:curvature-identity}
 \tau+\Ree J'(1)=1+d\cdot\kappa(\zeta).\qedhere
\end{equation}
\end{proof}

A strictly positive boundary margin must remain for polynomial approximation. Evaluating \labelcref{eq:blaschke-modulus} at $w_a$, with $r=1-\varepsilon$, gives
\begin{equation}\label{eq:poisson-gain}
 \Lambda_{1-\varepsilon}
 =1+\varepsilon\frac{1-|w_a|^2}{|1-w_a|^2}
   +O(\varepsilon^2).
\end{equation}
Since $|w_a|<1$, the coefficient is positive. It is the Poisson kernel at $w_a$ and the boundary point $1$.

\subsection{A family extending the cubic example}\label{subsec:binomial-family}

The cubic belongs to a family obtained by taking inverse images under $z\mapsto z^m$, $m\ge2$, of the disks
\[
 \{u:|u+iA|<\sqrt{1+A^2}\},\qquad A>0.
\]
Each disk contains $0$ and has $1$ on its boundary. The inverse image has $m$ lobes connected near the origin; the connections narrow as $A$ increases. We can look for normal returns in this explicit family without moving the boundary point $1$.

The M\"obius map
\[
 v\longmapsto\frac{v}{1+iA-iAv}
\]
takes $\D$ onto the displayed disk and fixes $1$. Composing with $v=w^m$ and taking a compatible $m$th root gives the following conformal coordinate.

\begin{proposition}[Conformal coordinates for binomial lemniscates]\label{prop:binomial}
Let $m\ge2$ be an integer, $A>0$, and $R=\sqrt{1+A^2}$. Define
\begin{equation}\label{eq:binomial-q}
 q_{m,A}(z)=\frac{z^m+iA}{R},\qquad
 \Omega_{m,A}=\{z:|q_{m,A}(z)|<1\}.
\end{equation}
Then $\Omega_{m,A}$ is an analytic Jordan domain, and
\begin{equation}\label{eq:binomial-phi}
 \Phi(w)=w(1+iA-iAw^m)^{-1/m}
\end{equation}
is a conformal map from $\D$ onto $\Omega_{m,A}$. Choose the analytic root whose value is $1$ at $w=1$. The map extends to a neighborhood of $\ol\D$, and
\begin{equation}\label{eq:binomial-boundary-data}
 \Phi(1)=1,\qquad \Phi'(1)=1+iA,\qquad
 \frac{\Phi''(1)}{\Phi'(1)}=(m+1)iA.
\end{equation}
Every critical point of $q_{m,A}$ lies in $\Omega_{m,A}$.
\end{proposition}

\begin{proof}
The calculation in \cref{prop:cubic-coordinates} applies with $D(w)=1+iA-iAw^m$. The zeros have modulus $(R/A)^{1/m}>1$, so the chosen root is analytic across $\ol\D$ and is a function of $w^m$ alone. From $\Phi(w_1)=\Phi(w_2)$, take $m$th powers to get $w_1^m=w_2^m$. The root factors are then equal, so $w_1=w_2$. Moreover,
\[
 \Phi'(w)=(1+iA)D(w)^{-1-1/m},
\]
which is nonzero and yields \labelcref{eq:binomial-boundary-data}. The identity
\begin{equation}\label{eq:binomial-modulus}
 |q_{m,A}(\Phi(w))|^2-1
 =\frac{|w|^{2m}-1}{|1+iA-iAw^m|^2}
\end{equation}
and the star-shapedness of $\Omega_{m,A}$ give, as in the cubic case, $\Phi(\D)=\Omega_{m,A}$. Finally, $q_{m,A}'$ has its only zero at $0$, where $|q_{m,A}(0)|=A/R<1$.
\end{proof}

An interior normal-return point satisfying the curvature inequality yields, by \cref{prop:model}, the required holomorphic model. The degree-$29\,984$ example in \cref{sec:explicit-29984} starts from a seed of degree $16$ in this family. We lower the degree with low-degree corrections, several high powers, and an independent normalization point.

\section{From boundary models to polynomials}\label{sec:faber}

The cubic construction used a perturbation whose values are small on the
closed domain and whose derivative grows rapidly outside. For an arbitrary analytic Jordan domain, Faber polynomials supply the
same device. Their remainder estimates give a lower bound for the moduli of all
normalized critical-point quotients and the boundary minimum minus an error
tending to zero. This will finish the proof
of \cref{thm:normal-main}.

Let $\Omega$ be an analytic Jordan domain, and let $\widehat\C=\C\cup\{\infty\}$ denote the extended complex plane. The exterior conformal map of $\Omega$ is the unique conformal map
\[
 \Psi:\widehat\C\setminus\ol\Omega\rightarrow
 \widehat\C\setminus\ol\D,
 \qquad \Psi(z)=\gamma z+O(1),\quad\gamma>0.
\]
Let $\psi=\Psi^{-1}$. We built the model in \cref{sec:geometry} with the interior coordinate $W$. For the polynomial perturbation, we need the exterior coordinate $\Psi$. Both map the boundary to $\T$, but they usually assign different angles to its points. The choice $\gamma>0$ fixes the rotation of $\Psi$.

We have $|\Psi|=1$ on the boundary and $|\Psi|>1$ outside. Thus the powers $\Psi^n$ have the exterior growth we need. Retaining only their nonnegative Laurent powers at infinity gives the $n$th \emph{Faber polynomial}, denoted by $H_n$. For example, if
$
 \Psi(z)=\gamma z+c_0+c_1z^{-1}+O(z^{-2}),
$
then
\[
 H_1(z)=\gamma z+c_0,\qquad
 H_2(z)=\gamma^2z^2+2\gamma c_0z+c_0^2+2\gamma c_1.
\]
In general, $H_n$ has degree $n$ and leading coefficient $\gamma^n$; see \cite{Curtiss,TohTrefethen}. Because the boundary is analytic, the discarded Laurent terms are exponentially small near $\partial\Omega$. We will also bound $H_n$ inside the domain.

\subsection{Polynomial parts and the exponential remainder}\label{subsec:faber-remainder}

The analytic boundary allows $\psi=\Psi^{-1}$ to extend injectively and holomorphically to $\{|w|>r_*\}$ for some $r_*\in(0,1)$. Define
\[
 \Gamma_s=\psi(\{|w|=s\}),\qquad s>r_*,
\]
and let $K_s$ be the compact set enclosed by $\Gamma_s$. Thus $\Gamma_1=\partial\Omega$ and $K_1=\ol\Omega$. In the collar and exterior, we can therefore define
\[
 E_n=H_n-\Psi^n.
\]
Once $E_n$ is bounded near the boundary, the maximum modulus principle bounds the polynomial $H_n$ on the whole enclosed domain.

\begin{lemma}[Exponential remainder in a fixed collar]\label{lem:faber-remainder}
Fix $r_*<\rho<s<1$. In the exterior of $\Gamma_\rho$,
\begin{equation}\label{eq:faber-cauchy}
 E_n(z)=\frac1{2\pi i}\int_{\Gamma_\rho}
             \frac{\Psi(\xi)^n}{\xi-z}\dd\xi.
\end{equation}
For every fixed $j\ge0$, there is a constant $C_j$ such that
\begin{equation}\label{eq:faber-error}
 |E_n^{(j)}(z)|\le C_j\rho^n
 \qquad(z\text{ on or outside }\Gamma_s).
\end{equation}
Moreover,
\begin{equation}\label{eq:faber-inside}
 \|H_n\|_{K_1}=O(1),\qquad
 \|H_n\|_{K_s}+n^{-1}\|H_n'\|_{K_s}=O(s^n).
\end{equation}
For every fixed $a\in\Omega$, choosing $s<1$ sufficiently close to $1$ gives
\begin{equation}\label{eq:faber-a}
 |H_n(a)|+n^{-1}|H_n'(a)|=O(s^n).
\end{equation}
\end{lemma}

\begin{proof}
Take $R$ large enough that $\Gamma_R$ encloses $z$. The Laurent expansion gives
\[
 H_n(z)=\frac1{2\pi i}\int_{\Gamma_R}
              \frac{\Psi(\xi)^n}{\xi-z}\dd\xi.
\]
When $z$ is outside $\Gamma_\rho$, deforming $\Gamma_R$ to $\Gamma_\rho$ crosses the pole at $\xi=z$, with residue $\Psi(z)^n$. Subtracting the residue gives \labelcref{eq:faber-cauchy}.

The curves $\Gamma_\rho$ and $\Gamma_s$ are separated by a positive distance. Differentiation under the integral gives
\[
 |E_n^{(j)}(z)|\le
 \frac{j!\,\operatorname{length}(\Gamma_\rho)}{2\pi\,
        \dist(\Gamma_\rho,\C\setminus\operatorname{int}K_s)^{j+1}}
 \rho^n
\]
on or outside $\Gamma_s$. The decay comes from $|\Psi(\xi)|=\rho<1$ on the contour.

On $\Gamma_1$, the identity $H_n=\Psi^n+E_n$ gives $|H_n|\le1+O(\rho^n)$. On $\Gamma_s$ it gives $H_n=O(s^n)$ and
\[
 n^{-1}H_n'=\Psi^{n-1}\Psi'+n^{-1}E_n'=O(s^n),
\]
where the constants may depend on $s$. The maximum modulus principle extends these bounds to $H_n$ and $H_n'$ in the enclosed domains. Finally, $a\in\operatorname{int}K_s$ for $s$ sufficiently close to $1$, which proves \labelcref{eq:faber-a}.
\end{proof}

The separation of scales from \cref{sec:cubic} now holds on the general domain. The perturbation $(H_n-H_n(a))/n$ is $O(n^{-1})$ on $\ol\Omega$, and its derivative is exponentially small on compact subsets of $\Omega$. Near the boundary, differentiation cancels the factor $1/n$, leaving $\Psi^{n-1}\Psi'$ to balance $f'$.

\subsection{The boundary inequality passes to polynomials}\label{subsec:basic-transfer}

The models in \cref{sec:geometry} and the approximation result \cref{lem:approximation} give normalized polynomials with nonvanishing derivative on $\ol\Omega$ and boundary minimum greater than $1$. We now transfer this inequality to every critical point.

\begin{theorem}[Faber transfer]\label{thm:faber}
Let $a\in\Omega$, and let $f$ be a polynomial satisfying
\[
 f(a)=0,\qquad f'(a)=1,\qquad f'\ne0\quad\text{on }\ol\Omega.
\]
Define
\[
 Q_n(z)=f(z)+\frac{H_n(z)-H_n(a)}n,
 \qquad L=\min_{z\in\partial\Omega}\left|\frac{f(z)}{z-a}\right|.
\]
For all sufficiently large $n$, $Q_n$ has degree $n$ and simple critical points, and
\begin{equation}\label{eq:faber-basic-distance}
 \max_{Q_n'(c)=0}\dist(c,\partial\Omega)=O(n^{-1}),
 \qquad \Smal(Q_n,a)\ge L-O(n^{-1}).
\end{equation}
For some $\eta\in(0,1)$, uniformly over all critical points,
\begin{equation}\label{eq:faber-critical-value}
 Q_n(c)-f(c)
 =-\frac{\Psi(c)f'(c)}{n\Psi'(c)}+O(\eta^n/n).
\end{equation}
The constants depend on the fixed domain, polynomial, and base point. Thus, if $L>1$, the construction gives counterexamples in every sufficiently large degree.
\end{theorem}

\begin{proof}
For $n>\deg f$, the leading coefficient of $H_n$ gives $\deg Q_n=n$. Choose $r_*<\rho<s<1<R$ so that $a\in\operatorname{int}K_s$ and $f'$ has no zeros on $K_R$. Such a choice is possible since $f'$ has no zeros on the compact set $\ol\Omega$. A fixed collar lies between $\Gamma_s$ and $\Gamma_R$, on which the exterior coordinate is defined.

By \cref{lem:faber-remainder}, $H_n'/n=O(s^n)$ on $K_s$. The positive minimum of $|f'|$ there ensures that $Q_n'$ has no zeros on $K_s$ for large $n$. In the collar,
\begin{equation}\label{eq:faber-derivative}
 Q_n'(z)=f'(z)+\Psi(z)^{n-1}\Psi'(z)+\frac{E_n'(z)}n.
\end{equation}
On $\Gamma_R$, the high-power term grows like $R^{n-1}$ while the other terms remain uniformly bounded. Thus $Q_n'(\Gamma_R)$ has the same winding number about $0$ as $(\Psi^{n-1}\Psi')(\Gamma_R)$. To compute a winding number, divide the total change of argument by $2\pi$. It is unchanged if the function is nonzero on the contour and the perturbation is smaller there than its modulus.

Under $\Psi$, the curve $\Gamma_R$ goes once counterclockwise around $|w|=R$. For $\Psi'$, holomorphy and nonvanishing in the exterior, together with the limit $\gamma>0$ at infinity, give winding number zero. The winding number of $Q_n'(\Gamma_R)$ is consequently $n-1$. Applied to the polynomial $Q_n'$, the argument principle gives $n-1$ zeros inside $\Gamma_R$, counted with multiplicity. This exhausts the degree. The earlier exclusion of $K_s$ confines every critical point to the fixed collar.

At a critical point $c$, \labelcref{eq:faber-derivative} gives
\[
 \Psi(c)^{n-1}
 =-\frac{f'(c)+E_n'(c)/n}{\Psi'(c)}.
\]
The right side is uniformly bounded above and away from zero in the collar. Hence $(n-1)\log|\Psi(c)|=O(1)$, so $|\Psi(c)|=1+O(n^{-1})$. The point
\[
 \zeta_c=\psi\!\left(\frac{\Psi(c)}{|\Psi(c)|}\right)
\]
lies on $\partial\Omega$. The derivative $\psi'$ is bounded on the fixed annulus. Comparing along a radial segment therefore gives $|c-\zeta_c|=O(n^{-1})$, the distance estimate in \labelcref{eq:faber-basic-distance}. Differentiating again and using the critical equation gives
\[
 Q_n''(c)=-(n-1)\frac{f'(c)\Psi'(c)}{\Psi(c)}+O(1).
\]
The coefficient of $n-1$ is bounded away from zero in modulus, so all critical points are simple for large $n$.

Multiplying the critical equation by $\Psi(c)/\Psi'(c)$ and substituting into $Q_n$ gives
\begin{align}\label{eq:faber-value-exact}
 Q_n(c)-f(c)
 &=-\frac{\Psi(c)f'(c)}{n\Psi'(c)}
   +\frac{E_n(c)}n
   -\frac{\Psi(c)E_n'(c)}{n^2\Psi'(c)}
   -\frac{H_n(a)}n.
\end{align}
The remainder estimates give \labelcref{eq:faber-critical-value}. By definition, $Q_n(a)=0$; also, $Q_n'(a)=1+O(\eta^n)$. Since the critical points stay a positive distance from $a$,
\[
 \left|\frac{Q_n(c)}{(c-a)Q_n'(a)}\right|
 =\left|\frac{f(c)}{c-a}\right|+O(n^{-1}).
\]
The quotient $f(z)/(z-a)$ has bounded derivative on the collar, so its modulus changes by $O(n^{-1})$ between $c$ and $\zeta_c$. Taking the minimum over the critical points gives the lower bound $L-O(n^{-1})$.
\end{proof}

The proof also applies uniformly to perturbations $f+\lambda_n(H_n-H_n(a))/n$ with $0<c\le|\lambda_n|\le C<\infty$. The critical equation again bounds $\lambda_n\Psi^n$ by the same quotient, so the leading critical-value correction is unchanged. For the cubic lemniscate, $\Psi^3=q$ and hence $H_{3N}=q^N$. The choices $n=3N$ and $\lambda_n=3$ recover $P_N$.

\subsection{Completing the normal-return construction}\label{subsec:normal-proof}

\begin{proof}[Proof of \cref{thm:normal-main}]
Choose $\Phi:\D\to\Omega$ with $\Phi(1)=\zeta$. \Cref{prop:curvature} shows that \labelcref{eq:normal-return} implies the hypotheses of \cref{prop:model}. Fix an admissible $r<1$ and set $g=\Lambda_r-1>0$.

Apply \cref{lem:approximation} with $\eta<g/3$ to obtain a normalized polynomial $f$ with nonvanishing derivative on $\ol\Omega$ and boundary minimum $L>1+2g/3$. By \cref{thm:faber},
\[
 \Smal(Q_n,a)\ge L-O(n^{-1}).
\]
For sufficiently large $n$, the error is at most $g/3$, giving $\Smal(Q_n,a)>1+g/3$. The normalisation
\[
 p_n(z)=\frac{Q_n(a+z)}{Q_n'(a)}
\]
gives $p_n(0)=0$, $p_n'(0)=1$ and $\Smal(p_n,0)=\Smal(Q_n,a)$. Taking $\delta=g/3$ proves the theorem.
\end{proof}

\noindent\emph{Size of the margin.}
For each admissible $r$, the proof gives
\[
 \delta=\frac{\Lambda_r-1}{3}
 =\frac13\left(\frac{|1-rw_a|}{|w_a-r|}-1\right),
 \qquad w_a=\Phi^{-1}(a).
\]
Any fixed $0<\delta<\Lambda_r-1$ follows by taking a more accurate approximation and a larger $n_0$.

Choose $\Phi$ with $\Phi(0)=a$ and $\Phi(1)=\zeta$. Then $w_a=0$ and $\Lambda_r=1/r$, so
\[
 \delta=\frac{1-r}{3r}
 =\frac{\varepsilon}{3(1-\varepsilon)},
 \qquad r=1-\varepsilon.
\]
The admissible size of $\varepsilon$ depends on the escape coefficient $\sigma=1+d\cdot\kappa(\zeta)>0$. With this normalisation, \cref{lem:splitting} and \labelcref{eq:curvature-identity} give constants $C_*>0$ and $\varepsilon_0\in(0,1)$ such that, for $0<\varepsilon<\varepsilon_0$, the only possible zeros of $N_{1-\varepsilon}$ in $\ol\D$ are the two roots near $1$, with
\[
 \left|\,|w_\pm|^2-1-2\sigma\varepsilon\,\right|
 \le C_*\varepsilon^{3/2}.
\]
For $0<\varepsilon<\varepsilon_0$, the condition $C_*\sqrt\varepsilon\le\sigma$ therefore gives $|w_\pm|^2\ge1+\sigma\varepsilon>1$. The constants reflect the remainder bounds and the exclusion of other zeros for the conformal map. Estimating them gives a numerical choice of $\delta$.

\section{The geometry of the critical points}\label{sec:sampling}

\Cref{sec:faber} places all critical points near the boundary and bounds
the mean-value functional from below. We now use the phase of the critical-point
equation to tell us how they are distributed along the boundary.
We obtain a first-order location formula and a limiting distribution,
with a critical point near every boundary point. We then show that the supremum of minima over admissible holomorphic models equals the universal constant $M$.

\subsection{Modulus locates the roots; phase distributes them}\label{subsec:phase-analysis}

As in the cubic construction of \cref{sec:cubic}, critical points arise from cancellation between $f'$ and the derivative of a high power. Omitting the exponentially small remainder in \labelcref{eq:faber-derivative} gives
\[
 f'(z)+\Psi(z)^{n-1}\Psi'(z)=0.
\]
With $D=n-1$ and $w=\Psi(z)$, it becomes
\[
 w^D=b(w),\qquad b(w)=-\psi'(w)f'(\psi(w)).
\]
On a sufficiently small fixed annulus about $\T$, $b$ is holomorphic and bounded in modulus above and away from zero. Every solution in this annulus therefore satisfies
\[
 D\log|w|=\log|b(w)|,
\]
so $|w|=1+O(D^{-1})$. This determines the radial location of the solutions. The angular location also depends on the argument of $b$.

For a continuous choice of argument, we seek a holomorphic logarithm $h$ of $b$, with $\Ree h=\log|b|$ and $\Imm h$ a continuous choice of argument. Such a logarithm need not exist on an annulus. For example, $b(w)=w$ gains $2\pi$ in argument on one circuit of $\T$. Recall that, in general, the total change of argument is $2\pi$ times the winding number
\[
 \wind(b(\T),0)
 =\frac1{2\pi i}\int_{\T}\frac{b'(w)}{b(w)}\dd w,
\]
and a holomorphic logarithm exists precisely when this integer is zero.

For $b=-\psi'(f'\circ\psi)$, each factor has zero winding number. The derivative $\psi'$ is holomorphic and nonzero for $|w|>r_*$ and tends to $\gamma^{-1}>0$ at infinity. Deforming $\T$ to a large circle therefore gives zero winding number. The curve $\psi(e^{it})$ traverses $\partial\Omega$ once counterclockwise and the argument principle gives zero winding for $f'\circ\psi$ because $f'$ has no zeros in $\Omega$. Winding numbers add under multiplication, and $-1$ contributes zero. Hence
\begin{equation}\label{eq:faber-zero-winding}
 \wind(b(\T),0)=0.
\end{equation}

Choose $h=\log b$. Writing $w=e^\xi$, the equation $w^D=b(w)$ becomes
\[
 D\xi=h(e^\xi)+2\pi ik,\qquad k\in\mathbb Z.
\]
Since $h$ is bounded on a smaller fixed annulus, $\xi$ is within $O(D^{-1})$ of $2\pi ik/D$. Exponentiation gives the equally spaced starting points $e^{2\pi ik/D}$. The Faber remainder adds a small holomorphic perturbation to $b$. The following lemma allows for this perturbation and counts all solutions in the annulus.

\begin{lemma}[Sampling a circle by a perturbed power equation]\label{lem:phase}
Let $b$ be holomorphic and nonzero on an annulus containing $\T$, and suppose its winding number on $\T$ is zero. Let $b_D$ be holomorphic there and converge to $b$ locally uniformly. On a smaller fixed annulus, for all sufficiently large integers $D$, we can choose holomorphic logarithms
\[
 h=\log b,\qquad h_D=\log b_D,
\]
with $h_D\to h$ locally uniformly and uniformly bounded derivatives. The equation
\begin{equation}\label{eq:phase-equation}
 w^D=b_D(w)
\end{equation}
has $D$ distinct simple solutions in a sufficiently small fixed annulus about $\T$, and no others there. They can be labelled by $k=0,\ldots,D-1$ so that, with $\omega_{D,k}=e^{2\pi ik/D}$,
\begin{equation}\label{eq:phase-expansion}
 w_{D,k}=\omega_{D,k}
 \exp\left(\frac{h_D(\omega_{D,k})}{D}+O(D^{-2})\right),
\end{equation}
uniformly in $k$.
\end{lemma}

\begin{proof}

Choose a holomorphic logarithm $h$ of $b$. On a smaller closed annulus about $\T$, the quotient $b_D/b$ converges uniformly to $1$. For large $D$, define
\[
 h_D=h+\operatorname{Log}(b_D/b),
\]
where $\operatorname{Log}$ is the branch near $1$ with $\operatorname{Log}(1)=0$. Then $e^{h_D}=b_D$ and $h_D\to h$ uniformly. Cauchy's estimate on a still smaller annulus gives uniform bounds for $h_D$ and $h_D'$.

For $w=e^\xi$, \labelcref{eq:phase-equation} is equivalent to $D\xi-h_D(e^\xi)\in2\pi i\mathbb Z$. Let $\xi_k=2\pi ik/D$. For each $k=0,\ldots,D-1$, solve
\begin{equation}\label{eq:phase-fixedpoint}
 \xi=\xi_k+\frac{h_D(e^\xi)}D.
\end{equation}
Call the right side $T_{D,k}(\xi)$. With $M\ge1$ a common bound for $|h_D|$, $T_{D,k}$ maps the disk $|\xi-\xi_k|\le2M/D$ into itself, since
$
 |T_{D,k}(\xi)-\xi_k|\le M/D.
$
For large $D$, these disks lie in a fixed vertical strip where the derivative of $h_D(e^\xi)$ is uniformly bounded by $C$. Hence
\[
 |T_{D,k}(\xi)-T_{D,k}(\widetilde\xi)|
 \le \frac{C}{D}|\xi-\widetilde\xi|
\]
for points in the same disk. Once $D>C$, the contraction mapping theorem gives a unique fixed point in each disk.

The resulting $w$ values are distinct even when contraction disks overlap. If the fixed points for $k$ and $\ell$ differed by $2\pi i j$, periodicity in \labelcref{eq:phase-fixedpoint} would give $k-\ell=Dj$. Since $0\le k,\ell<D$, this forces $k=\ell$. Each solution is simple, since
\[
 \frac{\mathrm d}{\mathrm dw}(w^D-b_D(w))
 =\frac{b_D(w)}w\bigl(D-wh_D'(w)\bigr)\ne0
\]
for sufficiently large $D$.

Conversely, take any solution $w$ in the smaller fixed annulus and choose $\xi$ with $e^\xi=w$. Then $D\xi=h_D(w)+2\pi i\ell$ for some $\ell\in\mathbb Z$. Express $\ell=k+Dj$ with $0\le k<D$, and replace $\xi$ by $\xi-2\pi ij$. The resulting logarithm satisfies \labelcref{eq:phase-fixedpoint} and $|\xi-\xi_k|\le M/D$, so it is the fixed point already found. There are no further solutions.

Substituting the bound $\xi-\xi_k=O(D^{-1})$ once more into \labelcref{eq:phase-fixedpoint} gives
\[
 \xi=\xi_k+D^{-1}h_D(\omega_{D,k})+O(D^{-2}),
\]
uniformly in $k$. Exponentiation gives \labelcref{eq:phase-expansion}.

\end{proof}

\subsection{Boundary sampling and harmonic measure}\label{subsec:critical-sampling}

Each point of an equally spaced grid in the exterior conformal angle has one nearby critical point. The resulting arclength density depends only on the domain and is generally nonuniform. The polynomial $f$ affects the first-order displacement through the function $h$ below.

\begin{theorem}[First-order sampling and limiting distribution]\label{thm:sampling}
Under the assumptions of \cref{thm:faber}, $Q_n$ has only simple critical points for sufficiently large $n$. Let $D=n-1$. There is a holomorphic branch
\begin{equation}\label{eq:h-definition}
 h(w)=\log\bigl(-\psi'(w)f'(\psi(w))\bigr)
\end{equation}
on a fixed annulus about $\T$, and the critical points can be labeled so that
\begin{equation}\label{eq:critical-sampling}
 c_{n,k}=\psi(\omega_{D,k})
 +\frac{\omega_{D,k}\psi'(\omega_{D,k})h(\omega_{D,k})}{D}
 +O(D^{-2}),\qquad 0\le k<D,
\end{equation}
uniformly in $k$. Their counting measures satisfy
\begin{equation}\label{eq:critical-measure}
 \frac1{n-1}\sum_{Q_n'(c)=0}\delta_c
 \ \rightharpoonup\ \mu_\Omega,
 \qquad
 \dd\mu_\Omega(z)=\frac{|\Psi'(z)|}{2\pi}\,|\dd z|
 \quad(z\in\partial\Omega).
\end{equation}
Here $\delta_c$ is the unit point mass at $c$. Weak convergence means convergence after integration against every bounded continuous function. The limiting measure $\mu_\Omega$ is the harmonic measure from infinity, or equivalently the logarithmic equilibrium measure of $\ol\Omega$.
\end{theorem}

\begin{proof}
Take the collar and derivative equation \labelcref{eq:faber-derivative} from the proof of \cref{thm:faber}. With $z=\psi(w)$, the critical-point equation becomes
\begin{equation}\label{eq:faber-bn}
 w^{n-1}=b_n(w),\qquad
 b_n(w)=-\psi'(w)\left(f'(\psi(w))+\frac{E_n'(\psi(w))}n\right).
\end{equation}
On a smaller fixed annulus, $b_n=b+O(\rho^n/n)$, where $b=-\psi'(f'\circ\psi)$ and $\rho<1$. The zero winding in \labelcref{eq:faber-zero-winding} allows us to apply \cref{lem:phase} with $D=n-1$. The $n-1$ solutions account for the entire critical set by \cref{thm:faber}.

The construction in the proof of \cref{lem:phase} gives $h_n=h+\operatorname{Log}(b_n/b)$, so $h_n-h=O(\rho^n/n)$. Substitute this into \labelcref{eq:phase-expansion} and expand $\psi$ to obtain \labelcref{eq:critical-sampling}, uniformly in $k$.

\smallskip\noindent
\emph{The limiting distribution.}
For every continuous function $g$ on a fixed compact neighborhood of $\partial\Omega$, the sampling formula and Riemann sums give
\begin{align*}
 \frac1D\sum_{k=0}^{D-1}g(c_{n,k})
 &-\frac1D\sum_{k=0}^{D-1}g(\psi(e^{2\pi ik/D}))\rightarrow0,\\
 \frac1D\sum_{k=0}^{D-1}g(\psi(e^{2\pi ik/D}))
 &\rightarrow\frac1{2\pi}\int_0^{2\pi}g(\psi(e^{it}))\dd t.
\end{align*}
The boundary parametrisation $z=\psi(e^{it})$ gives $|\dd z|=|\psi'(e^{it})|\dd t$ and $|\Psi'(z)|=|\psi'(e^{it})|^{-1}$. Hence uniform angular measure $\dd t/(2\pi)$ has arclength density $|\Psi'(z)|/(2\pi)$, proving \labelcref{eq:critical-measure}. By conformal invariance, this measure is the harmonic measure from infinity. By the Green-function characterization, it is also the equilibrium measure \cite{Ransford}.
\end{proof}

For nonempty compact sets $A,B\subset\C$, the Hausdorff distance is
\[
 d_H(A,B)=\max\left\{\sup_{a\in A}\dist(a,B),\
                         \sup_{b\in B}\dist(b,A)\right\},
 \qquad
 \dist(z,A)=\inf_{a\in A}|z-a|.
\]
The following corollary shows that every boundary point has a nearby critical point.

\begin{corollary}[The critical set samples the whole boundary]\label{cor:faber-limit}
Under the assumptions of \cref{thm:faber}, the Hausdorff distance between the critical set of $Q_n$ and $\partial\Omega$ is $O(n^{-1})$, and
\begin{equation}\label{eq:faber-limit}
 \Smal(Q_n,a)=L+O(n^{-1}).
\end{equation}
\end{corollary}

\begin{proof}
By \cref{thm:sampling}, each critical point lies $O(n^{-1})$ from $\psi(\omega_{n-1,k})$. The equally spaced points $\omega_{n-1,k}$ form an $O(n^{-1})$ net on $\T$, and $\psi$ has bounded derivative there. Hence every boundary point lies $O(n^{-1})$ from a critical point. The critical-value comparison in the proof of \cref{thm:faber} gives
\[
 \left|\frac{Q_n(c)}{(c-a)Q_n'(a)}\right|
 =\left|\frac{f(c)}{c-a}\right|+O(n^{-1}).
\]
The function $|f(z)/(z-a)|$ is Lipschitz on the collar, so every critical point gives a value at least $L-O(n^{-1})$. Conversely, a critical point within $O(n^{-1})$ of a boundary minimiser gives a value at most $L+O(n^{-1})$. These two bounds prove \labelcref{eq:faber-limit}.
\end{proof}

\subsection{Boundary minima and the universal constant}\label{subsec:boundary-limit}

For a normalized holomorphic model, define
\[
 \Boundary(F,\Omega,a)=
 \min_{z\in\partial\Omega}\left|\frac{F(z)}{z-a}\right|.
\]
The models in \cref{sec:geometry} have boundary minimum greater than $1$. Applying approximation and Faber transfer (\cref{lem:approximation,thm:faber}) to these models gives lower bounds for the universal constant $M$. In this section, we show the converse: for every polynomial with mean-value functional greater than $1$, there are admissible models whose boundary minima approach that value from below. Thus $M$ is a supremum over boundary models.

\begin{proposition}[Boundary formulation of the universal constant]\label{prop:boundary-universal}
The constant $M$ in \labelcref{eq:sharp-constants} satisfies
\[
 M=\sup_{\Omega,a,F}\Boundary(F,\Omega,a),
\]
where the supremum on the right is over analytic Jordan domains $\Omega$, points $a\in\Omega$, and functions $F$ holomorphic on a neighborhood of $\ol\Omega$ such that
\[
 F(a)=0,\qquad F'(a)=1,\qquad F'\ne0\quad\text{on }\ol\Omega.
\]
\end{proposition}

\begin{proof}
Write $M_{\partial}$ for the supremum on the right. 
For any admissible model and any $\eta>0$, \cref{lem:approximation} gives a normalised polynomial $f$ with no critical points on $\ol\Omega$ and
\[
 \Boundary(f,\Omega,a)>\Boundary(F,\Omega,a)-\eta.
\]
By \cref{cor:faber-limit}, the Faber perturbations satisfy
\[
 \Smal(Q_n,a)\rightarrow\Boundary(f,\Omega,a).
\]
Translating the base point and normalizing the derivative preserve this functional. Hence
\[
 M\ge\Boundary(f,\Omega,a)>\Boundary(F,\Omega,a)-\eta.
\]
Letting $\eta\downarrow0$ and taking the supremum gives $M\ge M_{\partial}$.

The model $F(z)=z-a$ gives $M_{\partial}\ge1$. Consider a normalised polynomial $p$ with $\Smal(p,0)>1$. Its divided-difference polynomial
\[
 G(z)=\frac{p(z)}z,\qquad G(0)=1,
\]
is nonconstant. On a level set $|G|=L$, the boundary quotient has constant modulus $L$. Choose
$
 1<L<\Smal(p,0)
$
such that $G'(z)\ne0$ whenever $|G(z)|=L$. This excludes only finitely many levels, the moduli of the critical values of $G$. Let $\Omega$ be the component of $\{|G|<L\}$ containing $0$.

Since $G$ is a nonconstant polynomial and $L$ is a regular level, $\{|G|=L\}$ is a finite union of disjoint analytic Jordan curves. Choose a boundary curve $\Gamma$ of $\Omega$. The maximum modulus principle gives $|G|<L$ throughout its bounded interior. By regularity, the interior is the side of $\Gamma$ occupied locally by $\Omega$. Thus $\Omega$ is an analytic Jordan domain.

Every critical point $c$ of $p$ satisfies
\[
 |G(c)|\ge\Smal(p,0)>L,
\]
so $p'$ has no zero on $\ol\Omega$. Thus $F=p$ is an admissible model with boundary quotient of constant modulus $L$. Hence $M_{\partial}\ge L$. Letting regular levels approach $\Smal(p,0)$ from below gives $M_{\partial}\ge\Smal(p,0)$. Polynomials with $\Smal(p,0)\le1$ satisfy the same inequality because $M_{\partial}\ge1$. The supremum over all normalized polynomials therefore gives $M_{\partial}\ge M$.
\end{proof}

\appendix
\section{An explicit counterexample of degree 29,984}\label{sec:explicit-29984}

The constructions in \cref{sec:cubic,sec:faber} give counterexamples in sufficiently large degrees. We now give an explicit polynomial of degree $29\,984$, based on the binomial domain from \cref{subsec:binomial-family}, with interval certificates for every critical-point quotient. All numerical enclosures use outward interval arithmetic with rigorous bounds on rounding errors. The accompanying certificate gives the full numerical bounds and arithmetic details:

\begin{center}\small
\url{https://github.com/MColbrook/Interval_arithmetic_counterexample_Smale_mean_value}
\end{center}\normalsize

In the existence proofs, increasing the perturbation degree eventually makes the error in the critical-point quotients smaller than the boundary margin. This can require a very large degree. To obtain the example below, we consider the critical points with the smallest quotients and adjust the polynomial. Low-degree corrections and lower powers in the perturbation allow us to improve these quotients. Moving the base point provides further freedom.

\subsection{How the adjustments reduce the degree}\label{subsec:explicit-mechanism}

We use the domain from \cref{subsec:binomial-family} with $m=16$, and write $\widehat q$ for its scaled polynomial coordinate of degree $16$. Let $a$ be the model anchor and $f$ a polynomial approximation to the holomorphic model. The degree enters through the balance between this approximation and the high power. Consider a single-power perturbation $U=f+\lambda\widehat q^N$, with $\lambda\ne0$. At a critical point $c$ with $\widehat q'(c)\ne0$, the equation $U'(c)=0$ gives
\[
 U(c)=f(c)-\frac{\widehat q(c)f'(c)}{N\widehat q'(c)}.
\]
In the power-transfer argument of \cref{sec:cubic}, a uniform bound on the remaining factor gives an $O(1/N)$ critical-value error. Increasing $N$ reduces the factor $1/N$, but raises the degree to $16N$ once $16N>\deg f$. At smaller $N$, changes in the critical values and displacements of the critical points can consume the boundary margin. We therefore seek corrections that bring the smallest quotients above $1$ without further increasing the degree.

For the final polynomial we take $N=1874$ and write $P=f+\Delta_0+X+H$, where $H$ is the high-degree perturbation. The first adjustment uses the low-degree terms $\Delta_0$ and $X$. The factors $(z-a)^2$ ensure that the model anchor's value and derivative remain unchanged. Both corrections have degree far below $16N=29\,984$, so the final degree does not increase.

If $P_s=P+sR$ and a simple critical point $c$ continues as $c_s$, then
\[
 \left.\frac{\mathrm dc_s}{\mathrm ds}\right|_{s=0}
 =-\frac{R'(c)}{P''(c)},\qquad
 \left.\frac{\mathrm d}{\mathrm ds}P_s(c_s)\right|_{s=0}=R(c).
\]
Thus the derivative of the correction moves the critical point to first order, while its value changes the critical value. This explains the form
\[
 \Delta_0=(z-a)^2(KR_0+T),\qquad K(z)=(z-\xi_1)^2(z-\xi_2)^2,
\]
with polynomial $R_0$ and distinct correction centres $\xi_1,\xi_2$. The double zeros of $K$ make $KR_0$ and its derivative vanish at the centers. The cubic $T$ can prescribe arbitrary values and first derivatives there. Once those four data are chosen, $R_0$ can be varied to improve the other quotients without disturbing them. The centers are fixed rational points near sensitive critical points. 

The second adjustment replaces a single high power by
\[
 H=\omega\widehat q^{N-8}\bigl(\gamma\widehat q^8+E\bigr),
 \qquad \gamma>0,\quad |\omega|=1,\quad \deg E\le127.
\]
A single power fixes the ratio of its value to its derivative. For $H=\lambda\widehat q^N$, this ratio is $\widehat q/(N\widehat q')$ wherever $H'\ne0$. The coefficients in $E$ let us change this ratio. With $N=1874$, the term $\widehat q^{1866}E$ has degree at most $29\,983$, one below the leading term $\omega\gamma\widehat q^{1874}$. The amplitude $\gamma$ and phase $\omega$ provide further parameters. We choose $X$, $E$, the amplitude, the phase, and the base point together to improve the minimum quotient.

The third adjustment is to choose the base point independently of the model anchor. We denote the chosen point by $\widehat a$. With $P$ fixed, moving the base changes the subtracted value, the distance to each critical point, and the derivative used for normalization. The critical points stay fixed. For a critical point $c$, define the quotient at a variable base point $u$ with $P'(u)\ne0$:
\[
 Q_c(u)=\frac{P(c)-P(u)}{(c-u)P'(u)}.
\]
Where $Q_c(u)\ne0$, its logarithmic derivative is
\begin{equation}\label{eq:explicit-base-gradient}
 \frac{Q_c'(u)}{Q_c(u)}
 =-\frac{P'(u)}{P(c)-P(u)}+\frac1{c-u}-\frac{P''(u)}{P'(u)}.
\end{equation}
A small displacement $\eta$ changes $\log|Q_c(u)|$ by $\Ree(\eta Q_c'(u)/Q_c(u))$ to first order. This gives two real parameters for adjusting the minimum, without changing the degree. 

For the final coefficients specified below, at the model anchor $a$, a critical point $c$ of $P$ satisfies
\[
 0.99982426177940
 <\left|\frac{P(c)-P(a)}{(c-a)P'(a)}\right|
 <0.99982426177944<1.
\]
The same polynomial therefore has mean-value functional below $1$ at $a$ and above $1$ at $\widehat a$. The normalization in \labelcref{eq:explicit-normalised} uses the full values $P(\widehat a)$ and $P'(\widehat a)$.

\subsection{The polynomial}\label{subsec:explicit-definition}

\begin{theorem}[An explicit polynomial]\label{thm:explicit-29984}
The polynomial $p\in\Q(i)[z]$ defined by \labelcref{eq:explicit-normalised}, with the accompanying rational coefficient tables, satisfies
\[
 \deg p=29\,984,\qquad p(0)=0,\qquad p'(0)=1,
\]
and
\begin{equation}\label{eq:explicit-margin}
 1.00000049595801<\Smal(p,0)<1.00000049595804.
\end{equation}
Its $29\,983$ critical points are all simple, and the minimum is attained at exactly one.
\end{theorem}

\begin{proof}
The finite formulas below give the degree and normalization. The certificate in \cref{subsec:explicit-certificate} establishes the critical-point count and the bounds in \labelcref{eq:explicit-margin}.
\end{proof}

We use the coordinate and truncation parameters
\begin{equation}\label{eq:explicit-coordinate}
\begin{gathered}
 m=16,\qquad N=1874,\qquad L=N-2=1872,\\
 A=\frac{1061603}{250000},\qquad C=1+A^2,\qquad S=A+1/A,\\
 \widehat q(z)=\frac{z^{16}+iA}{S},\qquad v(z)=i\widehat q(z).
\end{gathered}
\end{equation}
The coordinate used in \labelcref{eq:explicit-h} comes from the binomial family in \cref{subsec:binomial-family}. On $\Omega_{16,A}$, take principal fractional powers and, for a Blaschke parameter $r\in\D$, define
\[
 \beta=\left(\frac{1+iA}{C}\right)^{1/16},\qquad
 W(z)=\beta z(1+v(z))^{-1/16},\qquad
 B(w)=\frac{w-r}{1-\ol r w}.
\]
The coordinate satisfies
\[
 W(z)^{16}=\frac{(1+iA)z^{16}}{1+iAz^{16}};
\]
so $|W|=1$ on the model boundary $|z^{16}+iA|^2=C$. The rational coordinate satisfies $\widehat q=(A/\sqrt C)q_{16,A}$, and hence $|\widehat q|=A/\sqrt C<1$ on this boundary. Rationalizing $B(W)$ and truncating the binomial and geometric expansions gives \labelcref{eq:explicit-h}, with $b$ approximating $\beta$. The resulting $f(z)=(z-a)h_L(z)/h_L(a)$ is the first term of $P$ used above and approximates
\[
 F(z)=(z-a)\frac{B(W(z))}{B(W(a))}.
\]

The model anchor $a$, Blaschke parameter $r$, rational approximation $b$ to $\beta$, and correction centres $\xi_1,\xi_2$ have the exact values
\begingroup\small
\begin{align*}
 \Ree a&=1.0408093642746148137157525156090766412437,\\
 \Imm a&=0.3041902165017496705262264286568472099946,\\
 \Ree r&=0.9869461588252050688137472794964248376867,\\
 \Imm r&=-0.0432448562530681859520443550256337627761,\\
 \Ree b&=0.908850253183765306046968000840344747056738618388026628837241,\\
 \Imm b&=0.076267057870218869943792292811043863508717811372210553566224,\\
 \Ree\xi_1&=0.94945305488086338259789282088083367274676142809337,\\
 \Imm\xi_1&=0.01295609450414720552969704525015390110037343638189,\\
 \Ree\xi_2&=0.94845012497531595723791829728813409781718653511944,\\
 \Imm\xi_2&=0.01327429696527108235832258377560890558082991312789.
\end{align*}
\endgroup
Here $b$ is a coefficient parameter; we reserve $\widehat a$ for the normalization point.

Define
\[
 d_0=C+\ol r^{16}(-A^2+iA),\qquad
 d_1=iA-\ol r^{16}(1+iA),\qquad
 \alpha=-\frac{Cd_1}{iAd_0}.
\]
For $1\le k\le15$, set $\nu_k=1-k/16$ and compute the coefficients by the rational recurrence
\begin{equation}\label{eq:explicit-recurrence}
\begin{gathered}
 b_{k,0}=c_{k,0}=1,\\
 b_{k,j}=b_{k,j-1}\frac{\nu_k-j+1}{j},\qquad
 c_{k,j}=\alpha c_{k,j-1}+b_{k,j}\quad(1\le j\le L).
\end{gathered}
\end{equation}
Thus $b_{k,j}=\binom{\nu_k}{j}$. The resulting polynomial approximation to the Blaschke factor is
\begin{equation}\label{eq:explicit-h}
\begin{split}
 J(z)&=-rC(1+v(z))+\ol r^{15}(1+iA)z^{16},\\
 h_L(z)&=\frac{J(z)}{d_0}\sum_{j=0}^{L-1}(\alpha v(z))^j+\frac{(1-|r|^2)C}{d_0}
 \sum_{k=1}^{15}\ol r^{k-1}b^kz^k
 \sum_{j=0}^{L}c_{k,j}v(z)^j.
\end{split}
\end{equation}

For the corrections, use the basis $(5z/6)^k\widehat q(z)^\ell$ and set
\begin{equation}\label{eq:explicit-corrections}
\begin{aligned}
 K(z)&=(z-\xi_1)^2(z-\xi_2)^2,\qquad \mu=(\xi_1+\xi_2)/2,\\
 R_0(z)&=\sum_{\ell=0}^{63}\sum_{k=0}^{15}d_{\ell k}(5z/6)^k\widehat q(z)^\ell,\\
 T(z)&=\sum_{j=0}^{3}\eta_j(z-\mu)^j,\\
 \Delta_0(z)&=(z-a)^2\bigl(K(z)R_0(z)+T(z)\bigr),\\
 X(z)&=(z-a)^2\sum_{\ell=0}^{127}\sum_{k=0}^{15}x_{\ell k}(5z/6)^k\widehat q(z)^\ell,\\
 E(z)&=\sum_{\ell=0}^{7}\sum_{k=0}^{15}t_{\ell k}(5z/6)^k\widehat q(z)^\ell.
\end{aligned}
\end{equation}
The Gaussian-rational arrays $(d_{\ell k})$, $(\eta_j)$, $(x_{\ell k})$ and $(t_{\ell k})$ are specified in the accompanying coefficient tables.

Define the high-degree perturbation and the unnormalized polynomial by
\begin{equation}\label{eq:explicit-P}
\begin{gathered}
 \gamma=0.61480744530374742407,\qquad t=55.5578744900222005754,
 \qquad\omega=\frac{1+it}{1-it},\\
 H(z)=\omega\widehat q(z)^{1866}\bigl(\gamma\widehat q(z)^8+E(z)\bigr),\\
 P(z)=(z-a)\frac{h_L(z)}{h_L(a)}+\Delta_0(z)+X(z)+H(z).
\end{gathered}
\end{equation}
This gives $|\omega|=1$ exactly; the scalar $t$ is distinct from the entries $t_{\ell k}$. Finally, normalize at the independent base point by setting
\begin{equation}\label{eq:explicit-normalised}
\begin{gathered}
 \widehat a=1.04015231339331770677+0.302102329270430736141\,i,\\
 p(x)=\frac{P(\widehat a+x)-P(\widehat a)}{P'(\widehat a)}.
\end{gathered}
\end{equation}
Interval bounds establish $h_L(a)\ne0$ and $P'(\widehat a)\ne0$, so this defines a polynomial in $\Q(i)[x]$ with $p(0)=0$ and $p'(0)=1$.

The choice $L=N-2$ keeps the polynomial approximation below degree $16N$ in the leading power. The other terms fit below it too: the finite sums give
\[
\begin{array}{c|r}
 \text{term}&\text{degree, at most}\\\hline
 h_L&29\,967\\
 (z-a)h_L&29\,968\\
 \Delta_0&1029\\
 X&2049\\
 \widehat q^{1866}E&29\,983.
\end{array}
\]
Only $\omega\gamma\widehat q^{1874}$ contributes in degree $29\,984$. Its leading coefficient is $\omega\gamma/S^{1874}\ne0$. Translation and division by $P'(\widehat a)\ne0$ leave the degree unchanged.

\subsection{Certification of critical points and values}\label{subsec:explicit-certificate}

A critical point $c$ of $P$ becomes the critical point $c-\widehat a$ of $p$, and
\begin{equation}\label{eq:explicit-quotient}
 \left|\frac{p(c-\widehat a)}{c-\widehat a}\right|
 =\left|\frac{P(c)-P(\widehat a)}{(c-\widehat a)P'(\widehat a)}\right|.
\end{equation}
The certificate supplies $29\,983$ pairwise disjoint containing disks with rational centers and radii. Within each disk, a smaller disk contains exactly one zero of $P'$, counted with multiplicity. We prove this by comparison with a nonzero linear function.

\begin{lemma}[A local linear comparison]\label{lem:explicit-rouche}
Let $G$ and $g$ be holomorphic on a neighborhood of the closed disk $D=\{|z-z_0|\le\rho\}$. Suppose
\[
 E\ge\sup_D|G-g|,\qquad
 V\ge\sup_D|g'(z)-g'(z_0)|.
\]
If
\begin{equation}\label{eq:explicit-rouche}
 |g(z_0)|+\rho V+E<\rho|g'(z_0)|,
\end{equation}
then $G$ has exactly one zero in the open disk, counted with multiplicity.
\end{lemma}

\begin{proof}
On $|z-z_0|=\rho$, integrate along a radius to obtain
\[
 |G(z)-g'(z_0)(z-z_0)|
 \le E+|g(z_0)|+\rho V
 <|g'(z_0)(z-z_0)|.
\]
Rouch\'e's theorem then compares $G$ with the nonzero linear polynomial $g'(z_0)(z-z_0)$, which has exactly one zero in the disk. The zero of $G$ is consequently simple.
\end{proof}

For $G=P'$, we either enclose the finite polynomial directly or use a local analytic approximation with rigorous error bounds. The analytic comparison needs a binomial remainder formula valid beyond the Taylor circle. For $0<\nu<1$ and $v\notin(-\infty,-1]$,
\begin{equation}\label{eq:explicit-binomial-remainder}
 T_L^\nu(v)=(1+v)^\nu-\sum_{j=0}^L\binom\nu j v^j=\frac{(-1)^L\sin(\pi\nu)}{\pi}\,v^{L+1}
 \int_0^1\frac{s^{L-\nu}(1-s)^\nu}{1+sv}\,\mathrm ds.
\end{equation}
For $|v|<1$, expand $(1+sv)^{-1}$ and apply Euler's beta integral to prove the identity; analytic continuation extends it to the slit plane. The integral weight has total mass $|\binom\nu{L+1}|$ after multiplication by $\sin(\pi\nu)/\pi$. Consequently
\[
 |T_L^\nu(v)|\le
 \left|\binom\nu{L+1}\right|
 \frac{|v|^{L+1}}{\sigma(v)},\qquad
 \sigma(v)=\min_{0\le s\le1}|1+sv|.
\]
The recurrence \labelcref{eq:explicit-recurrence} also gives the finite identity
\[
 \frac{(1+v)^{\nu_k}}{1-\alpha v}
 -\sum_{j=0}^Lc_{k,j}v^j
 =\frac{T_L^{\nu_k}(v)+\alpha c_{k,L}v^{L+1}}{1-\alpha v}.
\]
Together with the finite geometric remainder and an enclosure of $\beta-b$, these identities bound the approximation error in \labelcref{eq:explicit-h}. Differentiation and Cauchy's estimates give the derivative bounds needed in \cref{lem:explicit-rouche}. Each analytic comparison takes place in a region that avoids the branch cut and has a positive lower bound on every denominator.

Alternatively, evaluate the derivatives of the finite polynomial at the center and analytically bound their variation. Suppose $|P'-g|\le E_R$ on a disk of radius $R>\rho$, and $V$ bounds $|g'(z)-g'(z_0)|$ on the concentric disk of radius $\rho$. Cauchy's estimate and integration along a radius give
\[
 \sup_{|z-z_0|\le\rho}|P''(z)-P''(z_0)|
 \le V+\frac{2\rho E_R}{(R-\rho)^2}=:V_P.
\]
The inequality
\[
 |P'(z_0)|+\rho V_P<\rho|P''(z_0)|
\]
therefore proves one simple critical point by \cref{lem:explicit-rouche} with $G=g=P'$ and $E=0$.

The $29\,983$ isolating disks lie in containing disks that are pairwise disjoint by exact rational distance comparisons. The one-zero counts in the isolating disks sum to $\deg P'=29\,983$. Thus they account for every critical point, and all are simple.

For the normalized critical-point quotients, define
\[
 \mathcal Q(z)=\frac{P(z)-P(\widehat a)}{(z-\widehat a)P'(\widehat a)}.
\]
Each isolating disk avoids $\widehat a$, and finite interval evaluation bounds $P'(\widehat a)$ away from zero. On such a disk,
\[
 \mathcal Q'(z)=
 \frac{P'(z)-(P(z)-P(\widehat a))/(z-\widehat a)}
 {(z-\widehat a)P'(\widehat a)}.
\]
If $R_Q\ge\rho\sup_D|\mathcal Q'|$ and
$\ell_0\le|\mathcal Q(z_0)|\le u_0$, integration along a radius gives
\[
 \ell_0-R_Q\le|\mathcal Q(c)|\le u_0+R_Q
\]
for the critical point $c$ in the disk. An analytic enclosure at the center gives the same bound after adding the finite-polynomial approximation error. 

Over the complete critical set, the interval bounds give
\[
 1.00000049595801828345<\Smal(p,0)
 <1.00000049595803927317,
\]
which implies \labelcref{eq:explicit-margin}. Every normalized critical-point quotient except one has modulus greater than $1.0000004992776471$, which is above the displayed upper bound. Thus the remaining critical point is the unique minimizer.

\begingroup
\renewcommand{\bibliofont}{\normalsize}
\raggedbottom
\nocite{Smale1998Journal}
\bibliographystyle{amsplain}
\bibliography{boundary_geometry_mean_value}
\endgroup
\end{document}